\documentclass[preprint,review,10pt]{elsarticle}
\usepackage{bbm}
\usepackage{amsmath,amssymb,amsthm}
\usepackage{color}
\usepackage{xcolor}
\usepackage{graphicx}
\usepackage{enumerate}
\usepackage{ifpdf}
\usepackage{algorithmic}
\usepackage{float}
\usepackage{CJK}
\usepackage{listings}
\usepackage{bbm}
\usepackage{mathrsfs}
\usepackage{amsmath}
\usepackage{txfonts}
\usepackage{epstopdf}
\usepackage{cases}
\usepackage{subfigure}
\usepackage{enumerate}
\usepackage{lipsum}
\usepackage{ulem}
\usepackage{multirow}
\makeatletter
\def\ps@pprintTitle{%
 \let\@oddhead\@empty
 \let\@evenhead\@empty
 \def\@oddfoot{}%
 \let\@evenfoot\@oddfoot}
\makeatother
 \usepackage{graphicx}

\usepackage{amssymb}
\usepackage{amsthm}

\makeatletter\@addtoreset{equation}{section} \makeatother
\newtheorem{theorem}{Theorem}[section]
\newtheorem{lemma}{Lemma}[section]

\newtheorem{remark}{Remark}[section]
\newtheorem{definition}{Definition}[section]

\usepackage{geometry}
\begin{document}
\begin{frontmatter}
\title{Quasi-interpolation for high-dimensional function approximation\tnoteref{label1}}
\tnotetext[label1]{This work was supported by NSFC (No.12271002, 12101310), Yong Outstanding Talents of Universities of Anhui Province (No.2024AH020019), NSF of Jiangsu Province ( No.BK20210315), the Fundamental Research Funds for the Central Universities (No.30923010912), and the Doctoral Research Startup Project of University of South China (No.5524QD015).}
\author{Wenwu Gao\fnref{label2}}
 \ead{wenwugao528@163.com}
 \author{Jiecheng Wang\fnref{label3}}
 \ead{jxwjc05@163.com}
\author{Zhengjie Sun*\fnref{label4}}
\ead{sunzhengjie1218@163.com}
 \author{Gregory E.~Fasshauer\fnref{label5}}
 \ead{fasshauer@mines.edu}
 \address[label2]{School of Big Data and Statistics, Anhui University, Hefei, China}
 \address[label3]{School of Economic, Management and Law, University of South China, Hengyang, China.}
 \address[label4]{School of Mathematics and Statistics, Nanjing University of Science and Technology, Nanjing, China}
 \address[label5]{Department of Applied Mathematics and Statistics, Colorado School of Mines, Golden, CO 80401, USA}
 \cortext[cor1]{Corresponding author}
\begin{abstract} The paper proposes a   general quasi-interpolation scheme  for  high-dimensional function approximation. To facilitate error analysis, we view  our quasi-interpolation as a two-step procedure. In the first step, we approximate a target function by a purpose-built convolution operator (with an error term referred to as convolution error). In the second step, we discretize the underlying convolution operator using certain quadrature rule  at the given sampling data sites (with an error term called discretization error).  The final approximation error is obtained as an optimally balanced sum of these two errors, which in turn views  our quasi-interpolation as a regularization technique that balances  convolution error and discretization error.  As a concrete  example, we construct a sparse grid quasi-interpolation scheme for high-dimensional function approximation. Both theoretical analysis and numerical implementations provide evidence that our quasi-interpolation scheme is robust and is capable of mitigating the curse of dimensionality   for approximating high-dimensional functions.
\end{abstract}
\end{frontmatter}

\section{\textit{Introduction}}
Function approximation is a fundamental tool for understanding and harnessing data both from the industry and research communities. Note that there are many practical function approximation problems involving high dimensionality, for example,  high-dimensional integration stemming from physics and finance \cite{Lemieux}, stochastic differential equations \cite{Sloan4}, uncertainty quantification (UQ) \cite{Hennig}, reinforcement learning \cite{Montanelli}, and  some data mining problems \cite{Vipin}.  But classical function approximation tools  usually suffer from the curse of dimensionality for  high-dimensional function approximation problems \cite{Bellmann}. This is mainly due to the observation that the cost of representation and  computation of an approximation with a prescribed accuracy grows exponentially with the dimension of the  problem under consideration.  Therefore, it is meaningful and urgent to propose efficient approximation techniques (either from representation or computation consideration) for high-dimensional function such that we can break/mitigate the curse of dimensionality. Under such a driving force, there have been a lot of influential  tools including the superposition method  \cite{Khavinson}, the sparse grid approach \cite{BungartzandGriebel},  the sparsity-based approximations \cite{Adcock}, and many others \cite{Griebel}.

The superposition method represents a high-dimensional function as a superposition of several continuous functions with fewer variables \cite{Khavinson}. Moreover, it is closely related to the well-known Hilbert's $13$th problem \cite{Vitushkin}. 
Particularly, the   Kronecker-product  type approximation \cite{Beylkin} takes $m$ linear combinations of tensor-product univariate functions: \begin{equation}\label{Kronecker}
  Kf(x_1,x_2,\cdots,x_d)=\sum_{j=1}^mc_j\prod_{k=1}^df_{j,k}(x_k), \ c_j\in\mathbb{R},
  \end{equation} which makes it  more efficient and suitable for approximating high-dimensional anisotropic functions. However, it  requires solving some kind of minimization problem to get the final approximants \cite{Griebel}. This not only leads to a large amount of computation time but also non-unique representations (due to the fact that objective functionals may not be globally convex). To circumvent these limitations,  we shall construct an alternative scheme under the framework of quasi-interpolation, a basic tool for function approximation (e.g., \cite{BackusandGilbert}, \cite{BeatsonLight},  \cite{Buhmann, Buhmann1, BuhmannandDai}, \cite{BuhmannandJag}, \cite{Dung0},  \cite{Dung}, \cite{DungandThao}, \cite{GaoandWu5}, \cite{GaoandWu6}, \cite{GaoandGreg}, \cite{Grohs}, \cite{Speleers}, \cite{Usta}, \cite{Vainikko}).

 Let  $f|_{\bold X}$ be corresponding discrete function values  evaluated at the sampling centers $\bold X$ over a bounded domain $\Omega\subset\mathbb{R}^d$. Let $\Phi_h$ be a purposely constructed univariate kernel that converges  (in the distributional sense) to the Dirac delta distribution as $h$ tends to zero. Our quasi-interpolation takes a general (tensor-product) form
 \begin{equation}\label{generalform}
 Q_{\bold h}f(\bold x)=\sum_{\bold x_{\bold j}\in \bold X}f(\bold x_{\bold j})\mu_{\bold j}\prod_{k=1}^d\Phi_{h_k}(x_k-x_{j_k}),\ \ \bold h=(h_1,\cdots,h_d), \ \bold x_{\bold j}=(x_{j_1},\cdots,x_{j_d}),\  \bold j=(j_1,j_2,\cdots,j_d).
 \end{equation}
   Here  $\{\mu_{\bold j}\}$ are quadrature weights corresponding to sampling centers $\bold X$ such that $Q_{\bold h}f(\bold x)$ is a quadrature rule of the convolution operator $\int_{\Omega}f(\bold t)\prod_{k=1}^d\Phi_{h_k}(x_k-t_{k})d\bold t$ for a fixed point $\bold x\in \Omega$. The quasi-interpolation \eqref{generalform}  yields an approximant directly without the need to solve any minimization problem. In addition, it inherits fair properties (i.e., optimality, regularization, univariate structure, etc.) while circumventing limitations of the Kronecker-product  type approximation. 

The main contributions of the paper are two-fold.
Firstly, we propose a general quasi-interpolation framework for high-dimensional periodic function approximation.  We would like to stress that choices of tensor-product multiquadric (MQ) trigonometric kernel and sparse grid  lead to a concrete scheme with three salient properties. It is local and shape-preserving (Lemma $3.2$).  In addition, it includes well established trigonometric B-spline quasi-interpolation \cite{Stanley} as  a special case (with all shape parameters being zero).   Thus it  is more suitable for constrained approximation  requiring locality and shape preservation (such as CAD, CAM in geometric design). More importantly, it gives nice approximation results to both target function and high-order derivatives (Theorem $3.3$), which makes it more useful especially in cases that involve approximating high-order derivatives. Secondly, by coupling periodization technique \cite{NasdalaandPotts} together with quasi-interpolation, we construct a simple and efficient quasi-interpolation scheme for high-dimensional non-periodic function approximation over bounded domain   without the notorious boundary problem. We remind readers that periodization has been a well established and widely used technique for numerical integration (\cite{Kritzer}, \cite{Novak}) as well as non-periodic function approximation (\cite{NasdalaandPotts}, \cite{Peigney}).

The paper is structured as follows. In Section $2$, we introduce notations and definitions,  and review some important results of periodization strategies, multiquadric trigonometric function,  and  sparse grids.
Section $3$  elaborates the key idea  of  approximating high-dimensional function  through constructing a  sparse grid quasi-interpolation scheme  as a concrete example.  Subsection $3.1$ constructs a sparse grid quasi-interpolation scheme  for  periodic function approximation over a high-dimensional torus and derives its simultaneous  approximation orders. Subsection $3.2$ extends the quasi-interpolation scheme to the non-periodic case via the periodization technique. Section $4$ is devoted to an exposition of numerical simulation results. Conclusions and discussions are presented in Section $5$.
\label{Maintext}
\section{\textit{Preliminaries}}
\subsection{\textit{Notations and definitions}}
We focus on function approximations over two commonly considered domains \cite{NasdalaandPotts}: the torus $\mathbb{T}^d\simeq[0,1)^d$ and the unit  cube $[0,1]^d$. We say a function is defined on a torus $\mathbb{T}^d$ meaning that  the function is periodic with respect to each coordinate and the period $\bold T:=(1,1,\cdots,1)=\{1\}^d$.
Our target functions are from weighted Sobolev spaces of integer order \cite{Adams}. More precisely, let $ \Omega\in \{\mathbb{T}^d, \ [0,1]^d\}$ and let  $L_p(\Omega,\omega)$ denote the Banach spaces consisting of all measurable functions $f$ with respect to the positive weight $\omega$ on $\Omega$ for which $\|f\|_{L_p(\Omega,\omega)} < \infty$,  where $1 \le p \le \infty$. Here we define \begin{equation*}
\|f\|_{L_p(\Omega,\omega)} :=
\begin{cases}
{\displaystyle\left( \int_{\Omega} |f(x)|^p \omega(x)dx \right)^{1/p},}  & \text{if $1 \le p < \infty $,}
\\
\inf \{C: |f(x)\omega(x)| \le C \; \text{for almost all} \;  x \},  & \text{if $p=\infty$.}
\end{cases}
\end{equation*}
 Let $l$ be a positive integer  and $\mathbb{N}$  the set of natural number. Denote by $\mathbb{N}^d$ the $d$-fold tensor-product of $\mathbb N.$ Let $\boldsymbol{\alpha}=(\alpha_1,\cdots,\alpha_d)\in \mathbb{N}^d$ be a multi-index with $|\boldsymbol{\alpha}|=\sum^d_{j=1} \alpha_j$ and $\|\boldsymbol{\alpha}\|_{\infty}=\max\{\alpha_j,\ j=1,\cdots, d\}$. Denote by $D^{\boldsymbol{\alpha}}$ the differential operator of order $|\boldsymbol{\alpha}|$, that is, $$D^{\boldsymbol{\alpha}}f(\bold x)=\frac{\partial^{\alpha_1}}{\partial x_1^{\alpha_1}}\cdots\frac{\partial^{\alpha_d}}{\partial x_d^{\alpha_d}}f(x_1,\cdots,x_d).$$ 
For   $\Omega=[0,1]^d$,  we consider  weighted Sobolev spaces
$ W^{l}_p (\Omega,\omega)$  of functions satisfying
$$D^{\boldsymbol{\alpha}} f \in L_p(\Omega,\omega), \quad 0 \leq \|\boldsymbol{\alpha}\|_{\infty} \leq l,$$ whose weighted Sobolev norms are defined as 
\[
\|f\|_{W^{l}_p({\Omega},\omega)}=\Big(\sum_{0 \le \|\boldsymbol{\alpha}\|_{\infty} \le l} \|D^{\boldsymbol{\alpha}} f\|_{L_p({\Omega},\omega)}^p\Big)^{1/p}.
\]
 For $\Omega=\mathbb{T}^d$ and a constant weight  $\omega$,  we   follow the definition of  Sobolev spaces (Wiener spaces) introduced in Kolomoitsev et al., \cite{Kolomoitsev3}  in terms of the Fourier coefficients  $$\hat{f}(\bold k)=\int_{\mathbb{T}^d}f(\bold x)e^{-i(\bold k,\bold x)}d\bold x,$$
 where $(\bold k,\bold x)=\sum_{j=1}^dx_jk_j$.
 In particular, we adopt the definitions of $ W^l_2 (\mathbb{T}^d)$  (the Hilbert space $\mathbb{H}^l(\mathbb{T}^d)$) and $ W^l_{\infty} (\mathbb{T}^d)$ from \cite{Kolomoitsev3}:
   \begin{equation*}\label{H}
   \mathbb{H}^l(\mathbb{T}^d):=\Bigg\{f\in L_2(\mathbb{T}^d): \|f\|_{\mathbb{H}^l(\mathbb{T}^d)}=\Bigg(\sum_{\bold k\in \mathbb{Z}^d}(1+|\bold k|)^{2l}|\hat{f}(\bold k)|^2\Bigg)^{\frac{1}{2}}<\infty \Bigg\},
  \end{equation*}
 and 
  \begin{equation*}\label{H}
    \mathbb{W}_{\infty}^l(\mathbb{T}^d):=\Bigg\{f\in L_1(\mathbb{T}^d): \|f\|_{\infty}=\sup_{\bold k\in \mathbb{Z}^d}(1+|\bold k|)^{l}|\hat{f}(\bold k)|<\infty \Bigg\}.
 \end{equation*}
\subsection{\textit{Periodization}}
Periodization is essentially a special trick of changes of variables  \cite{Kritzer}. For more details of periodization, we refer readers to   \cite{Kuo}, \cite{Potts}, and the references therein. Given a \textbf{non-periodic} function $g\in W_p^l([0,1]^d,\omega)$, periodization means that we transform $g$ into a \textbf{periodic} function $f$ that can be continuously extendable on the torus $\mathbb{T}^d$ by a torus-to-cube transformation $\boldsymbol{\gamma}(\bold x)=(\gamma_1(x_1),\cdots,\gamma_d(x_d))$  such that
$$\int_{[0,1]^d}|g(\bold y)|^p\omega(\bold y)d\bold y=\int_{[0,1]^d}|g(\boldsymbol{\gamma}(\bold x))|^p\omega(\boldsymbol{\gamma}(\bold x))d\boldsymbol{\gamma}(\bold x) =\int_{\mathbb{T}^d}|f(\bold x)|^pd\bold x.$$ The transformation preserves the $L_p$-norm ($1\leq p<\infty$). In addition, the periodic function $f$ can be expressed by \cite{Kritzer}
\begin{equation}\label{periodicfunction}
f(\bold x)=g(\gamma_1(x_1),\cdots,\gamma_d(x_d))\prod_{j=1}^d\Bigg(\omega_j(\gamma_j(x_j))\gamma_j'(x_j)\Bigg)^{1/p},\ x_j\in \mathbb{T}.
 \end{equation}
 This implies that we respectively adopt the tensor-product forms of weight function $\boldsymbol{\omega}$ and $\boldsymbol{\gamma}$ to apply the  periodization strategy to each coordinate \cite{NasdalaandPotts}, namely,  $$\boldsymbol{\gamma}(\bold x)=(\gamma_1(x_1),\cdots,\gamma_d(x_d))=\prod_{j=1}^d\gamma_j(x_j),\quad \boldsymbol{\omega}(\boldsymbol{\gamma}(\bold x))=\prod_{j=1}^d\omega_j(\gamma_j(x_j)).$$ Under such a premise, we only need to give the definition of a torus-to-cube transformation for $d=1$. We introduce some results in Nasdala and Potts \cite{NasdalaandPotts}.
\begin{definition}\label{transform}
We call a mapping $\gamma_{j}:[0,1]\rightarrow [0,1]$ satisfying  $\lim_{x_{j}\rightarrow 0}\gamma_{j}(x_{j})=0$ and  $\lim_{x_{j}\rightarrow 1}\gamma_{j}(x_{j})=1$  a torus-to-cube transformation if it is continuously differentiable, increasing and has the first-order derivative $\gamma_{j}'\in C(\mathbb{T})$ (the collection of all continuous functions).
\end{definition}
   
Furthermore, to guarantee that the eventual periodic function $f$  is in  $\mathbb{H}^l(\mathbb{T}^d)$, both the transformation and weight function have to satisfy certain boundary conditions. In particular,   we recall the following sufficient   conditions (Theorem $4$ in Nasdala and Potts \cite{NasdalaandPotts}).
\begin{lemma}\label{lemmaH(T)}
Let a product transformation $\boldsymbol{\gamma}$, a product weight function $\boldsymbol{\omega}$,  and the corresponding transformed function $f$ be defined in Equation \eqref{periodicfunction}. Then for any function $g\in L_2([0,1]^d,\boldsymbol{\omega})\cap C_{\text{mix}}^l([0,1]^d)$, we have $f\in\mathbb{H}^l(\mathbb{T}^d)$ if the transformation satisfies
\begin{equation}\label{conditions}
\boldsymbol{\gamma}\in C_{\text{mix}}^l([0,1]^d)\quad\quad \text{and} \quad\quad D^{\boldsymbol{\alpha}}\Bigg[\prod_{j=1}^d\sqrt{\omega_j(\gamma_j)\gamma_j'}\Bigg]\in C_0([0,1]^d),
\end{equation}
for all multi-indices $\boldsymbol{\alpha}$ with $\|\boldsymbol{\alpha}\|_{\infty}\leq l$.
Here $C_{\text{mix}}^l([0,1]^d)$ denotes the collection of all functions with mixed continuous differentiability of order $l$ and $C_{0}([0,1]^d)$ denotes the set of all continuous functions vanishing at boundaries.
\end{lemma}

From  Lemma \ref{lemmaH(T)},  one can observe that  derivatives of the transformation  $\gamma_j$  vanish at boundaries. The sine transformation (see Formula $(26)$ in  \cite{NasdalaandPotts}) is a  typical torus-to-cube transformation. Moreover, Nasdala and Potts \cite{NasdalaandPotts} also introduced a   family of   parameterized torus-to-cube transformation, for example, logarithmic transformation (Formula $(24)$ in \cite{NasdalaandPotts}), error function transformation (Formula $(25)$ in \cite{NasdalaandPotts}).   
\subsection{\textit{MQ trigonometric function}}
MQ trigonometric function   $\sqrt{c^2+\sin^2 x/2}$ was firstly constructed  in reference \cite{GaoandWu1} to get a smooth and periodic kernel  by smoothing away all discontinuous points of  the second-order trigonometric kernel $|\sin x/2|$. Later, reference \cite{SunandGao} provided a novel   interpretation by regarding MQ trigonometric function as a restriction (up to a constant) of the well known  MQ kernel (proposed by Hardy \cite{Hardy})  to the unit circle, which allows to construct more versatile smooth and periodic kernels from radial kernels.   We adopt a scaled version  to get a $1$-periodic MQ trigonometric function defined over  $\mathbb{T}$ as \begin{equation}\label{multiquadric}
 \phi_c(x)=\frac{1}{2\pi}\sqrt{c^2+\sin^2\pi x}, \ x\in \mathbb{T}.
  \end{equation} Here $c\leq e^{-1}$ is a small positive shape parameter and the rescaling coefficient $\frac{1}{2\pi}$  is to ensure that Inequality \eqref{integral} holds. Moreover, the scaled version naturally leads to a tensor-product kernel \eqref{tensorkernel} defined over torus $\mathbb{T}^d$ satisfying Inequality \eqref{erortoone}, which warrants that scheme \eqref{general} is a quasi-interpolant defined over $\mathbb{T}^d$. MQ trigonometric function carries over many fair properties of MQ kernel. In particular, we have the following two lemmas (adapted from reference \cite{GaoandWu2}).
\begin{lemma}\label{highorderofphi}
Let $\ell$ be any given nonnegative integer and $\phi_c^{(\ell)}(x)$ be the  $\ell$-th order derivative with respect to $x$. Then, there is a constant $C(\ell)$ depending only on $\ell$  such that  
\begin{equation}\label{eq:lem_highorderofphi1}
    \left|\phi_c^{(\ell)}( x)\right|\leq C(\ell)c^{1-\ell},  \ \   \sin \pi x\leq c,
\end{equation}
and
\begin{equation}\label{eq:lem_highorderofphi2}
    \left|\phi_c^{(\ell)}(x)\right|\leq  C(\ell)\sin^{1-\ell} \pi x, \ \ \sin \pi x\geq c.
\end{equation}
In particular, when $\ell=2$, we have that the inequality
\begin{equation}\label{eq:lem_highorderofphi3}
    \int_{\mathbb{T}}|\phi_c''(x)|dx\leq (24+12\ln 2)|\ln c|
\end{equation}
holds true for any $c\leq e^{-1}$.
\end{lemma}
\begin{proof}
Note that $\phi_c(x)$ can be viewed as a composite function $g_c(h(x))=\sqrt{c^2+h(x)}$ with $h(x)=\sin^2\pi x$. Based on Fa\`{a} di Bruno's formula \cite[Thm.~2]{Roman}, we have 
$$\phi_c^{(\ell)}(x)=\frac{1}{2\pi}\sum \frac{\ell !}{i_1!\cdots i_{\ell}!}(g_c^{(i)})(h(x))\Big(\frac{h'(x)}{1!}\Big)^{i_1}\cdots \Big(\frac{h^{(\ell)}(x)}{\ell!}\Big)^{i_{\ell}}.$$
Here $i=i_1+\cdots+i_{\ell}$ and the summation is taken all over $i_1,\cdots i_{\ell}$, for which $i_1+2i_2+\cdots +\ell i_{\ell}=\ell$. 
Consequently,  Inequalities \eqref{eq:lem_highorderofphi1} and \eqref{eq:lem_highorderofphi2} can be derived from Lemma $3.1$ in \cite{GaoandWu2} with $C(\ell)=(2\pi)^{\ell-1}\sum \frac{\ell!}{i_1!\cdots i_{\ell}!}$. Here we only need to prove Inequality \eqref{eq:lem_highorderofphi3}.

  Observe that $1\leq |\ln c|$ for $c\leq e^{-1}$.  We first split the above integral into two parts to get 
    \begin{equation*}
    \begin{aligned}
        \int_{\mathbb{T}}|\phi_c''(x)|dx&=\int_{\sin\pi x\leq c, x\in \mathbb{T}}|\phi_{c}''(x)|dx+\int_{\sin\pi x\geq c, x\in \mathbb{T}}|\phi_{c}''(x)|dx\\
        &\leq 6\pi\Bigg(\int_{\sin\pi x\leq c, x\in \mathbb{T}}\frac{1}{c}dx+\int_{\sin\pi x\geq c,, x\in \mathbb{T}}\frac{1}{\sin\pi x}dx\Bigg)\\
        &\leq 6\Bigg(\int_{\sin\mu\leq c,\mu\in [0,\pi)}\frac{1}{c}d\mu+\int_{\sin\mu\geq c,\mu\in [0,\pi)}\frac{1}{\sin\mu}d\mu\Bigg)\\
        &\leq 12+6\int_{\sin\mu\geq c,\mu\in [0,\pi)}\frac{1}{\sin\mu}d\mu\\
         &\leq 12|\ln c|+6\int_{\sin\mu\geq c,\mu\in [0,\pi)}\frac{1}{\sin\mu}d\mu.
    \end{aligned}   
    \end{equation*}  
    
    Moreover, we have 
    \begin{equation}\label{bound}
        \begin{split}
        \int_{ \sin \mu\geq c,\mu\in [0,\pi)}\frac{1}{\sin \mu}d\mu
        &=\int_{ \arcsin c}^{\pi-\arcsin c}\frac{1}{\sin \mu}d\mu=\frac{1}{2}\ln\Big(\frac{1-\cos \mu}{1+\cos \mu}\Big)\Big |_{ \arcsin c}^{\pi-\arcsin c}\\
        &=\ln\frac{1+\sqrt{1-c^2}}{1-\sqrt{1-c^2}}= \ln\frac{(1+\sqrt{1-c^2})^2}{c^2}\\
            &\leq2(\ln 2-\ln c)\leq 2(1+\ln 2) |\ln c|.
        \end{split}
    \end{equation}
    Combining above results  completes the proof.
\end{proof}
\begin{lemma}\label{convolution}
   Let $\Psi_c(x)=\phi_c^{''}(x)+\pi^2\phi_c(x)$. Then we have
 \begin{equation}\label{integral}
 0\leq \int_{\mathbb{T}}\Psi_c(x)dx-1
 \leq \frac{(1+|\ln c|)c^2}{2}.
 \end{equation}
  Moreover,  if we define  $$\bold{C}_{\Psi_c}(f)(x)=f*\Psi_c(x)=\int_{\mathbb{T}}f(t)\Psi_c(x-t)dt,$$ then there is a constant $C_0>0$ such that
\begin{equation}\label{con}
\left\|\bold{C}_{\Psi_c}(f)-f\right\|_{\infty}\leq C_0\cdot c^2(c^2+1)|\ln c|\end{equation}
holds true for any $f\in C^2(\mathbb{T})$.
\end{lemma}
\begin{proof}
    Observe that
  \begin{equation*}
        \begin{split}
    \int_{\mathbb{T}}\Psi_c(x)dx-1&=\int_{\mathbb{T}}\phi_c^{''}(x)+\pi^2\phi_c(x)dx\\
&=\pi^2\int_{\mathbb{T}}\phi_c(x)dx\\
        &=\frac{\pi}{2}\int_{\mathbb{T}}\sqrt{c^2+\sin^2\pi x}dx\\
           &=\frac{1}{4}\int_{0}^{2\pi}\sqrt{c^2+\sin^2t/2}dt.
        \end{split}
    \end{equation*}   
     The lemma follows directly from  Lemma $3.2$ and Lemma $3.3$ in reference \cite{GaoandWu2}.
    \end{proof}
\subsection{\textit{Sparse grids}}

We introduce sparse grids on the hypercube  $[0,1]^{d}$ (see \cite{Garcke1}).
 Let $\mathbf{l}=(l_{1},l_{2},\cdots ,l_{d})\in \mathbb{N}^{d}$ be a multi-index  and  $ \mathbb{W}_{\mathbf{l}}$ be a directionally uniform grid in $[0,1]^{d}$ with the mesh size $h_{\mathbf{l}}:=2^{-\mathbf{l}}=(2^{-l_{1}},2^{-l_{2}},\cdots 2^{-l_{d}})$.
By applying the union operator to    $\mathbb{W}_{\mathbf{l}}$, we define the $d$-dimensional sparse grid $\mathbb{W}_{n,d}$ at level $n \geqslant 1$  in the form
\begin{equation}\label{subgrid}
\mathbb{W}_{n,d}=\bigcup_{|\mathbf{l}|=n +d-1}^{}\mathbb{W}_{\mathbf{l}}.
 \end{equation}
%
%

   Obviously, there is much redundancy in the above definition since subgrids $\{\mathbb{W}_{\mathbf{l}}\}$  are not disjoint. Regarding this, the number of grid points in the sparse grid is calculated via
  \begin{equation}\label{numberofpoints}
  |\mathbb{W}_{n,d}|=\sum_{k=0}^d(-1)^{k+d-1}\binom{d-1}{k}\sum_{|\mathbf{l}|=n+k}|\mathbb{W}_{\mathbf{l}}|.
  \end{equation}  Figure $1$ shows a toy example of full grid and sparse grid under $n=7,d=2$. We can see that the sparse grid can reduce a large amount of grid points than its full counterpart, which may be one reason why the sparse grid based approach can mitigate the curse of dimensionality for high-dimensional function approximation.
\begin{figure}[!htp]
\centering
\subfigure[Full grids]{
\begin{minipage}[t]{0.4\linewidth}
\centering
\includegraphics[scale=0.3]{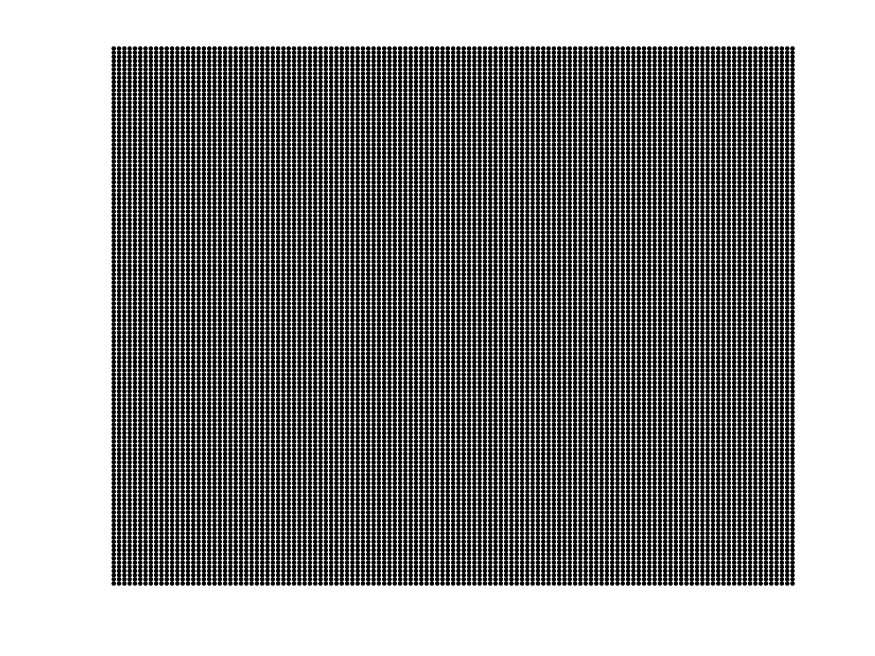}
\end{minipage}%
}%
\subfigure[Sparse grids]{
\begin{minipage}[t]{0.4\linewidth}
\centering
\includegraphics[scale=0.3]{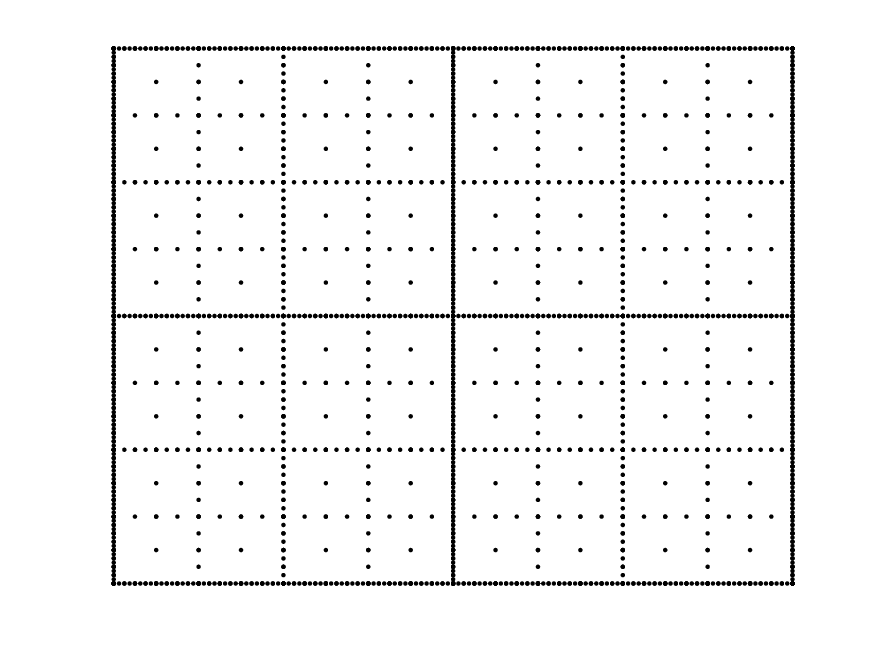}
\end{minipage}%
}%
\caption{ Full grid and sparse grid at level 7 on $[0,1]^{2}$, with $129^{2}=16641$ and $1281$ grid points, respectively}
\end{figure}
\section {\textit{The main result}}
 In   quasi-interpolation literature \cite{BuhmannandJag}, to facilitate theoretical analysis, one often studies quasi-interpolation  in the whole space $\mathbb{R}^d$. However, data in real world approximation problems are often sampled over a bounded domain, say $[0,\ 1]^d$. If we apply the quasi-interpolation (defined over $\mathbb{R}^d$) directly to the data, then it always yields the well-known boundary problem, that is, the resulting quasi-interpolation provides poor approximation at the boundaries of the bounded domain. Boundary extension  has been a widely used method for circumventing   boundary problem \cite{BeatsonPowell}, \cite{Jeong}, \cite{WuandSchaback}.  Another approach  adopts   error correction technique either by constructing iterated quasi-interpolation \cite{Fasshauer4} or by constructing multilevel quasi-interpolation \cite{Franz}, \cite{Usta}.  By introducing a periodization technique \cite{NasdalaandPotts}  into quasi-interpolation, we will provide a simple  approach  (especially for high-dimensional cases)  transferring non-periodic function approximation to periodic function approximation without introducing boundary problems. 
 
 Construction   of our quasi-interpolation defined over  hypercube $[0,\ 1]^d$ consists of three steps. We first transform a \textbf{non-periodic} approximand $g$  defined over $[0,\ 1]^d$  into a \textbf{periodic} function $f$ defined over the torus $\mathbb{T}^d$ via the torus-to-cube transformation $\boldsymbol{\gamma}$ (introduced in Subsection $2.2$). Then we focus on constructing a (sparse grid) quasi-interpolation scheme $Q_{n,d}f$ (see Equation \eqref{sparsequasi}) of $f$ over the torus $\mathbb{T}^d$. Finally, based on $Q_{n,d}f$, we construct a corresponding quasi-interpolation $Qg$ (see Equation \eqref{nonperiodic}) of $g$ over the hypercube $[0,\ 1]^d$.
 \begin{remark}
  In keeping the  exposition transparent, we demonstrate our construction process through constructing sparse grid quasi-interpolation as a concrete example. However, the basic idea is rather general and goes through for many other quasi-interpolation schemes. In addition, based on the embedding inequality $\|\cdot\|_{p,\Omega}\leq |\Omega|^{1/p}\|\cdot\|_{\infty,\Omega}$, for any $1\leq p<\infty$ and bounded domain $\Omega$, we only consider the  $L_{\infty}$-norm  approximation error. 
 \end{remark}
\subsection{Sparse grid quasi-interpolation for approximating periodic functions defined on $\mathbb{T}^d$}
Fourier series expansion is  a famous and powerful  tool for periodic function approximation, see e.g.,   \cite{Kaarnioja}, \cite{Kammerer}, \cite{Morrow}, \cite{NasdalaandPotts}, \cite{Peigney},  \cite{Potts}, \cite{Potts1}, \cite{Sloannew}. The basic idea is to first approximating an (unknown) periodic function using partial sum (with finite terms) of its Fourier series expansion and then evaluating corresponding Fourier coefficients numerically based on the given periodic data. However, as pointed out by Peigney \cite{Peigney}, the number of Fourier coefficients in the partial sum grows exponentially with the dimension of the approximation problem.  Hubbert et al., \cite{Hubbert} proposed a sparse grid Gaussian convolution approximation, which avoids computing a large amount of Fourier coefficients due to the fast decay of Gaussian kernel.
Kolomoitsev and his coauthors (\cite{Kolomoitsev0}, \cite{Kolomoitsev1}) proposed efficient approaches for periodic function approximation. They first applied the periodic extension to a scaled non-periodic kernel satisfying Strang-Fix conditions  to construct a periodic kernel. Then they constructed approximation with the periodic kernel under the  framework of   quasi-projection operators thoroughly discussed by Kolomoitsev and Skopina \cite{Kolomoitsev2}. Corresponding error estimates were derived in terms of periodic Strang-Fix conditions satisfied by the periodic kernel. Recently,  Kolomoitsev et al., \cite{Kolomoitsev3} studied approximation properties of multivariate periodic functions from weighted Wiener spaces by sparse grid methods based on quasi-interpolation with trigonometric polynomials.

In this subsection, we provide a general quasi-interpolation approach for periodic function approximation by viewing quasi-interpolation as a two-step procedure  (first convolution and then discretization of the convolution \cite{GeneralizedWu}), which allows for more choices of kernels (e.g., trigonometric B-spline, MQ trigonometric function, the restriction of a radial function to the circle) and sampling centers (e.g., uniform centers, directionally uniform centers, lattices, low-discrepancy centers, random  centers). 
 As a concrete example, we choose the tensor-product MQ trigonometric  kernel  and the sparse grid  leading to a local and shape-preserving quasi-interpolant (Lemma $3.2$) that can alleviate the curse of dimensionality while giving nice approximation results to both target function and  high-order derivatives (Theorem $3.3$).

We first construct a tensor-product multiquadric trigonometric kernel.  Applying the second-order trigonometric divided difference (with respect to the variable $t_k$ \cite{Koch}) to $\phi_{c_k}(x_k-t_k)$, we have \begin{equation}\label{newpsi}
\Psi_{c_k,h_k}(x_k):=[-h_k,0,h_k]_{Tr_2}\phi_{c_k}(x_k-t_k)=2\pi^2\frac{\phi_{c_k}(x_k+h_k)-2\cos(\pi h_k)\phi_{c_k}(x_k)+\phi_{c_k}(x_k-h_k)}{\sin (2\pi h_k)\sin(\pi h_k)},\ k=1,2,\cdots,d.
\end{equation}
Note that $\Psi_{c_k,h_k}(x_k)$ is a (trigonometrically) \textbf{discretized version}    of $\Psi_{c_k}(x_k)=(D_{x_k}^2+\pi^2I)\phi_{c_k}(x_k)$  (defined in Lemma \ref{convolution}) with a stepsize $h_k$.
Moreover, based on these univariate functions $\Psi_{c_k, h_k}$, we construct a multivariate periodic kernel  in the tensor-product form
\begin{equation}\label{tensorkernel}
\Psi_{\mathbf{c},\mathbf{h}}(\bold x)=\prod_{k=1}^d\Psi_{c_k,h_k}(x_k), \ \bold x=(x_1,\cdots,x_d)\in \mathbb{T}^d, \ \mathbf{c}=(c_1,\cdots,c_d), \ \bold h=(h_1,\cdots,h_d).
\end{equation}
Furthermore, we can derive the following fair properties as given in Lemma \ref{integralvalues} and Theorem \ref{convolutionerror}.
\begin{lemma}\label{integralvalues}
Let $\Psi_{\mathbf{c},\mathbf{h}}$ be defined as in Equation \eqref{tensorkernel}.  Then, there is a constant $C_{1}>0$ such that
\begin{equation}\label{erortoone}
\int_{\mathbb{T}^d}\Psi_{\mathbf{c},\mathbf{h}}(\bold x)d\bold x-1\leq C_{1}\sum_{k=1}^d(c_k^2+h_k^2)|\ln c_k|
\end{equation}
holds for all sufficiently small $c_k\leq e^{-1}$  and $h_k$.
\end{lemma}
\begin{proof}
The tensor-product structure of $\Psi_{\mathbf{c}, \mathbf{h}}$ together with  the Fubini's theorem leads to
$$\int_{\mathbb{T}^d}\Psi_{\mathbf{c},\mathbf{h}}(\bold x)d\bold x =\prod_{k=1}^d\int_{\mathbb{T}}\Psi_{c_k,h_k}(x_k)dx_k.$$
We begin with deriving the  bounds of $$\int_{\mathbb{T}}\Psi_{c_k,h_k}(x_k)dx_k, \ k=1,2,\cdots, d.$$

Observing that $\Psi_{c_k,h_k}$ is a discretized version of $\Psi_{c_k}=(D_{x_k}^2+\pi^2I)\phi_{c_k}$, we have
\begin{equation*}
\begin{split}
\int_{\mathbb{T}}\Psi_{c_k,h_k}(x_k)dx_k&=4\pi^2\frac{1-\cos(\pi h_k)}{\sin(2\pi h_k) \sin(\pi h_k)}\int_{\mathbb{T}}\phi_{c_k}(x_k)dx_k\\
&\leq (\pi^2+B_kh_k^2)\int_{\mathbb{T}}\phi_{c_k}(x_k)dx_k
\end{split}
\end{equation*}
for  constants $B_k>0$, $k=1,2,\cdots, d$. This together with $$\int_{\mathbb{T}}\phi_{c_k}(x_k)dx_k=\frac{1}{\pi^2}\int_{\mathbb{T}}(D_{x_k}^2+\pi^2I)\phi_{c_k}(x_k)dx_k$$leads to
\begin{equation*}
\int_{\mathbb{T}}\Psi_{c_k,h_k}(x_k)dx_k\leq \left(1+\frac{B_k}{\pi^2}h_k^2\right)\int_{\mathbb{T}}\Psi_{c_k}(x_k)dx_k.
\end{equation*}
Furthermore, based on Inequality \eqref{integral} and $1\leq |\ln c_k|$ for small $c_k\leq e^{-1}$, we have
\begin{equation*}  
\begin{split}
\int_{\mathbb{T}}\Psi_{c_k,h_k}(x_k)dx_k&\leq \left(1+\frac{B_k}{\pi^2}h_k^2\right)\left(1+\frac{1}{2}(|\ln c_k|+1)c_k^2\right)\\
&\leq \left(1+\frac{B_k}{\pi^2}h_k^2\right)\left(1+c_k^2|\ln c_k|\right)\leq \left(1+\frac{B_k}{\pi^2}h_k^2\right)\left(1+(c_k^2+h_k^2)|\ln c_k|\right)\\
&\leq 1+(c_k^2+h_k^2)|\ln c_k|+\frac{B_k}{\pi^2}(c_k^2+h_k^2)|\ln c_k|+\frac{B_k}{\pi^2}(c_k^2+h_k^2)|\ln c_k|\\
&\leq 1+C_k^*(c_k^2+h_k^2)|\ln c_k|, \quad \text{where}\ C_k^*=1+2B_k/\pi^2.
\end{split}
\end{equation*} 
Therefore, by letting $\alpha_k:=C_k^*(c_k^2+h_k^2)|\ln c_k|$, we obtain
\begin{equation*}
\begin{split}
\int_{\mathbb{T}^d}\Psi_{\mathbf{c},\mathbf{h}}(\bold x)d\bold x
&\leq\prod_{k=1}^d\Bigg(1+\alpha_k\Bigg)\leq 1+\sum_{j=1}^d\Bigg(\sum_{k=1}^d\alpha_k\Bigg)^j\\
&\leq 1+\frac{(\sum_{k=1}^d\alpha_k)\big[1-(\sum_{k=1}^d\alpha_k)^{d}\big]}{1-\sum_{k=1}^d\alpha_k}\\
&\leq 1+\frac{\sum_{k=1}^d\alpha_k}{1-\sum_{k=1}^d\alpha_k}.
\end{split}
\end{equation*}
Since $c_k$ and $h_k$ are sufficiently small, one can find a  positive constant $\tilde{C}$ such that $1-\sum_{k=1}^d\alpha_k\geq \tilde{C}>0$, which in turn leads to
$$\int_{\mathbb{T}^d}\Psi_{\mathbf{c},\mathbf{h}}(\bold x)d\bold x\leq 1+\frac{1}{\tilde{C}}\sum_{k=1}^d\alpha_k=1+C_1\sum_{k=1}^d(c_k^2+h_k^2)|\ln c_k|,\quad \text{where} \ C_1=\max\limits_k\{C_k^*\}/\tilde{C}.$$
\end{proof}
%

With this anisotropic tensor-product multiquadric  trigonometric kernel $\Psi_{\bold c,\bold h}$, we construct a convolution operator in the form $$\bold{C}_{\Psi_{\bold c,\bold h}}(f)(\bold x)=f*\Psi_{\bold c,\bold h}(\bold x)=\int_{\mathbb{T}^d}f(\bold t)\Psi_{\bold c,\bold h}(\bold x-\bold t)d\bold t.$$
Furthermore, corresponding bounds of  simultaneous approximation errors to both the function and its derivatives can be derived  in the following theorem.
\begin{theorem}\label{convolutionerror}
Let $\bold{C}_{\Psi_{\bold c,\bold h}}(f)$ be defined as above. Let $r$ be a nonnegative integer and $f\in W_{\infty}^{r+2}(\mathbb{T}^d)$. Then, there exists some positive constant $C_{2}$ that is independent of $\bold c$ and $\bold h$ such that the inequality
\begin{equation}\label{simcon}
\Bigg\|D^{\boldsymbol{\alpha}}f-D^{\boldsymbol{\alpha}}\bold{C}_{\Psi_{\bold c,\bold h}}(f)\Bigg\|_{\infty}\leq C_{2}\sum_{k=1}^d(c_k^2+h_k^2)|\ln c_k|
\end{equation}
holds true for any multi-index $\boldsymbol{\alpha}=(\alpha_1,\alpha_2,\cdots,\alpha_d)$ satisfying $0\leq \|\boldsymbol{\alpha}\|_{\infty}\leq r$.
\end{theorem}
\begin{proof}
Due to the periodicity of $f$ with respect to each coordinate, we can rewrite \begin{equation*}
\begin{split}
D^{\boldsymbol{\alpha}}f(\bold x)-D^{\boldsymbol{\alpha}}\bold{C}_{\Psi_{\bold c,\bold h}}(f)(\bold x)&=D^{\boldsymbol{\alpha}}f(\bold x)-\bold{C}_{D^{\boldsymbol{\alpha}}\Psi_{\bold c,\bold h}}(f)(\bold x)\\
&=f^{(\boldsymbol{\alpha})}(\bold x)-\int_{\mathbb{T}^d}f(\bold t)D^{\boldsymbol{\alpha}}\Psi_{\bold c,\bold h}(\bold x-\bold t)d\bold t\\
&=f^{(\boldsymbol{\alpha})}(\bold x)-\int_{\mathbb{T}^d}f^{(\boldsymbol{\alpha})}(\bold t)\Psi_{\bold c,\bold h}(\bold x-\bold t)d\bold t
\end{split}
\end{equation*} using $|\boldsymbol{\alpha}|$ times integration by parts. In addition, observing that $\Psi_{c_k,h_k}$ is a trigonometric discretization of $\Psi_{c_k}$ with a step size $h_k$ \cite{GaoandWu2}, we can express $\Psi_{c_k,h_k}$  in terms of $\Psi_{c_k}$ as $$\Psi_{c_k,h_k}=\Psi_{c_k}+\bar{B_k}h_k^2\phi^{(4)}_{c_k}(\eta_k),\ \eta_k\in (x_k-h_k,x_k+h_k),$$ for some constant $\bar{B_k}$.  This together with Lemma \ref{integralvalues} leads to
\begin{equation*}
\begin{split}
&|D^{\boldsymbol{\alpha}}f(\bold x)-D^{\boldsymbol{\alpha}}\bold{C}_{\Psi_{\bold c,\bold h}}(f)(\bold x)|\\
\leq &\Bigg|\int_{\mathbb{T}^d}(f^{(\boldsymbol{\alpha})}(\bold x)-f^{(\boldsymbol{\alpha})}(\bold t))\Psi_{\bold c,\bold h}(\bold x-\bold t)d\bold t\Bigg|+C_1\sum_{k=1}^d(c_k^2+h_k^2)|\ln c_k|\\
=&\Bigg|\int_{\mathbb{T}^d}(f^{(\boldsymbol{\alpha})}(\bold x)-f^{(\boldsymbol{\alpha})}(\bold t))\prod_{k=1}^d\Bigg[\Psi_{c_k}(x_k-t_k)+\bar{B_k}h_k^2\phi_{c_k}^{(4)}(\eta_k-t_k)\Bigg]d\bold t\Bigg|+C_1\sum_{k=1}^d(c_k^2+h_k^2)|\ln c_k|.
\end{split}
\end{equation*}
Thus it suffices to show that the inequality
\begin{equation}\label{induction}
\Bigg|\int_{\mathbb{T}^d}(f^{(\boldsymbol{\alpha})}(\bold x)-f^{(\boldsymbol{\alpha})}(\bold t))\prod_{k=1}^d\Bigg[\Psi_{c_k}(x_k-t_k)+\bar{B_k}h_k^2\phi_{c_k}^{(4)}(\eta_k-t_k)\Bigg]d\bold t\Bigg|\leq \bar{C}\sum_{k=1}^d(c_k^2+h_k^2)|\ln c_k|
 \end{equation} holds true uniformly for any $x_k\in \mathbb{T}$ and $\eta_k\in (x_k-h_k, x_k+h_k)$. Here $\bar{C}$ is some positive constant.

We adopt  mathematical induction with respect to the dimension $d$ and start with the univariate case.

 For $d=1$, it is easy to get the triangle inequality
\begin{equation*}
\begin{split}&\Bigg|\int_{\mathbb{T}}(f^{(\alpha_1)}(x_1)-f^{(\alpha_1)}(t_1))\Bigg[\Psi_{c_1}(x_1-t_1)+\bar{B_1}h_1^2\phi_{c_1}^{(4)}(\eta_1-t_1)\Bigg]dt_1\Bigg|
\\\leq& \Bigg|\int_{\mathbb{T}}(f^{(\alpha_1)}(x_1)-f^{(\alpha_1)}(t_1))\Psi_{c_1}(x_1-t_1)dt_1\Bigg|+\bar{B_1}h_1^2\Bigg|\int_{\mathbb{T}}(f^{(\alpha_1)}(x_1)-f^{(\alpha_1)}(t_1))\phi_{c_1}^{(4)}(\eta_1-t_1)dt_1\Bigg|.
\end{split}
\end{equation*}
Moreover, according to Equation \eqref{con} of Lemma \ref{convolution}, the first part of the right-hand side of the above inequality can be bounded by:
$$\Bigg|\int_{\mathbb{T}}(f^{(\alpha_1)}(x_1)-f^{(\alpha_1)}(t_1))\Psi_{c_1}(x_1-t_1)dt_1\Bigg|\leq C_0 c_1^2(c_1^2+1)|\ln c_1|.$$
To get the bound of the second part, we first get
\begin{equation*}
\begin{split}
\int_{\mathbb{T}}f^{(\alpha_1)}(x_1)\phi_{c_1}^{(4)}(\eta_1-t_1)dt_1&=f^{(\alpha_1)}(x_1)\int_{\mathbb{T}}\phi_{c_1}^{(4)}(\eta_1-t_1)dt_1\\
&=f^{(\alpha_1)}(x_1)\Big( \phi_{c_1}^{(3)}(\eta_1-1)-\phi_{c_1}^{(3)}(\eta_1)\Big)=0,
\end{split}
\end{equation*} due to the $1$-periodicity of $\phi_{c_1}$. Then,  applying the integration by parts twice to the definite integral $\int_{\mathbb{T}}f^{(\alpha_1)}(t_1)\phi_{c_1}^{(4)}(\eta_1-t_1)dt_1$,
we have
\begin{equation*}
\Bigg|\int_{\mathbb{T}}f^{(\alpha_1)}(t_1)\phi_{c_1}^{(4)}(\eta_1-t_1)dt_1\Bigg|\\
\leq \|f^{(2+\alpha_1)}\|_{\infty} \int_{\mathbb{T}}|\phi_{c_1}^{(2)}(\eta_1-t_1)|dt_1.
\end{equation*} This together with Inequality \eqref{eq:lem_highorderofphi3} leads to 
\begin{equation*}
\Bigg|\int_{\mathbb{T}}f^{(\alpha_1)}(t_1)\phi_{c_1}^{(4)}(\eta_1-t_1)dt_1\Bigg|\leq  \bar{C}^*|\ln c_1|, \quad \text{where} \quad  \bar{C}^*=(24+12\ln 2)\|f^{(2+\alpha_1)}\|_{\infty}.
\end{equation*}
 Consequently, we have
$$\bar{B_1}h_1^2\Bigg|\int_{\mathbb{T}}(f^{(\alpha_1)}(x_1)-f^{(\alpha_1)}(t_1))\phi_{c_1}^{(4)}(\eta_1-t_1)dt_1\Bigg|\leq \bar{B_1}h_1^2 \bar{C}^* |\ln c_1|.$$ 

Finally, combining the above two parts yields 
\begin{equation*}
    \begin{split}
        &\Bigg|\int_{\mathbb{T}}(f^{(\alpha_1)}(x_1)-f^{(\alpha_1)}(t_1))\Bigg(\Psi_{c_1}(x_1-t_1)+B_1h_1^2\phi_{c_1}^{(4)}(\eta_1-t_1)\Bigg)dt_1\Bigg|\\
        \leq& C_0c_1^2(1+c_1^2)|\ln c_1|+ \bar{C}^* \bar{B}_1h_1^2|\ln c_1|
        \leq 2C_0c_1^2|\ln c_1|+\bar{C}^* \bar{B}_1h_1^2|\ln c_1|\\
        \leq& (2C_0+ \bar{C}^* \bar{B}_1)(c_1^2+h_1^2)|\ln c_1|,
    \end{split}
\end{equation*} which is a special case of Inequality \eqref{induction} with $d=1$ and $\bar{C}= (2C_0+ \bar{C}^* \bar{B}_1)$.

Suppose that  Inequality \eqref{induction} holds true for  $d-1$ with any $d\geq 2$. Namely, $$\Bigg|\int_{\mathbb{T}^{d-1}}(f^{(\boldsymbol{\alpha}_{d-1})}(\bold x_{d-1})-f^{(\boldsymbol{\alpha}_{d-1})}(\bold t_{d-1}))\prod_{k=1}^{d-1}\Bigg[\Psi_{c_k}(x_k-t_k)+B_kh_k^2\phi_{c_k}^{(4)}(\eta_k-t_k)\Bigg]d\bold t_{d-1}\Bigg|\leq \bar{C}\sum_{k=1}^{d-1}(c_k^2+h_k^2)|\ln c_k|,$$ where $\bold v_{d-1}$ denotes restriction of a $d$-variate vector $\bold v$ to its $(d-1)$-dimensional space.
Then,  we have
 \begin{equation*}
 \begin{split}
 &\Bigg|\int_{\mathbb{T}^d}(f^{(\boldsymbol{\alpha})}(\bold x)-f^{(\boldsymbol{\alpha})}(\bold t))\prod_{k=1}^d\Bigg[\Psi_{c_k}(x_k-t_k)+B_kh_k^2\phi_{c_k}^{(4)}(\eta_k-t_k)\Bigg]d\bold t\Bigg|\\
 \leq& \Bigg|\int_{\mathbb{T}^d}\Bigg(f^{(\boldsymbol{\alpha})}(\bold x_{d-1},x_d)-f^{(\boldsymbol{\alpha})}(\bold t_{d-1},x_d)\Bigg)\prod_{k=1}^d\Bigg[\Psi_{c_k}(x_k-t_k)+B_kh_k^2\phi_{c_k}^{(4)}(\eta_k-t_k)\Bigg]d\bold t_{d-1}dt_d\Bigg|\\
 &+\Bigg|\int_{\mathbb{T}^d}\Bigg(f^{(\boldsymbol{\alpha})}(\bold t_{d-1},x_d)-f^{(\boldsymbol{\alpha})}(\bold t_{d-1},t_d)\Bigg)\prod_{k=1}^d\Bigg[\Psi_{c_k}(x_k-t_k)+B_kh_k^2\phi_{c_k}^{(4)}(\eta_k-t_k)\Bigg]dt_dd\bold t_{d-1}\Bigg| \\
 \leq& \bar{C}\sum_{k=1}^{d-1}(c_k^2+h_k^2)|\ln c_k|\cdot\Bigg|\int_{\mathbb{T}}\Bigg[\Psi_{c_d}(x_d-t_d)+B_dh_d^2\phi_{c_d}^{(4)}(\eta_d-t_d)\Bigg]dt_d\Bigg|\\
 &+\Bigg|\int_{\mathbb{T}^{d-1}}\int_{\mathbb{T}}\Bigg(f^{(\boldsymbol{\alpha})}(\bold t_{d-1},x_d)-f^{(\boldsymbol{\alpha})}(\bold t_{d-1},t_d)\Bigg)\prod_{k=1}^d\Bigg[\Psi_{c_k}(x_k-t_k)+B_kh_k^2\phi_{c_k}^{(4)}(\eta_k-t_k)\Bigg]dt_dd\bold t_{d-1}\Bigg|.  \end{split}
 \end{equation*}
 Moreover, according to Inequality \eqref{erortoone} of Lemma \ref{integralvalues}, we have
 $$\Bigg|\int_{\mathbb{T}}\Bigg[\Psi_{c_d}(x_d-t_d)+B_dh_d^2\phi_{c_d}^{(4)}(\eta_d-t_d)\Bigg]dt_d\Bigg|\leq 1+C_1(c_d^2+h_d^2)|\ln c_d|,$$ which in turn leads to
  \begin{equation}\label{firstpart}
  \begin{split}
  &\Bigg|\int_{\mathbb{T}^d}\Bigg(f^{(\boldsymbol{\alpha})}(\bold x_{d-1},x_d)-f^{(\boldsymbol{\alpha})}(\bold t_{d-1},x_d)\Bigg)\prod_{k=1}^d\Bigg[\Psi_{c_k}(x_k-t_k)+B_kh_k^2\phi_{c_k}^{(4)}(\eta_k-t_k)\Bigg]d\bold t_{d-1}dt_d\Bigg|\\
  \leq& \Bigg(\bar{C}\sum_{k=1}^{d-1}(c_k^2+h_k^2)|\ln c_k|\Bigg)\Bigg(1+C_1(c_d^2+h_d^2)|\ln c_d|\Bigg).
  \end{split}
  \end{equation}
  Furthermore, based on the above starting case of mathematical induction (that is $d=1$), we have
 $$\Bigg|\int_{\mathbb{T}}\Bigg(f^{(\boldsymbol{\alpha})}(\bold t_{d-1},x_d)-f^{(\boldsymbol{\alpha})}(\bold t_{d-1},t_d)\Bigg)\Bigg[\Psi_{c_d}(x_d-t_d)+B_dh_d^2\phi_{c_d}^{(4)}(\eta_d-t_d)\Bigg]dt_d\Bigg|\leq (2C_0+\bar{C}^*\bar{B}_1)(c_d^2+h_d^2)|\ln c_d|.$$ This together with  Inequality \eqref{erortoone} of Lemma \ref{integralvalues} leads to
 \begin{equation}\label{secondpart}
 \begin{split}
 &\Bigg|\int_{\mathbb{T}^{d-1}}\int_{\mathbb{T}}\Bigg(f^{(\boldsymbol{\alpha})}(\bold t_{d-1},x_d)-f^{(\boldsymbol{\alpha})}(\bold t_{d-1},t_d)\Bigg)\prod_{k=1}^d\Bigg[\Psi_{c_k}(x_k-t_k)+B_kh_k^2\phi_{c_k}^{(4)}(\eta_k-t_k)\Bigg]dt_dd\bold t_{d-1}\Bigg|\\
 \leq& (2C_0+\bar{C}^*\bar{B}_1)(c_d^2+h_d^2)|\ln c_d|\cdot\Bigg|\int_{\mathbb{T}^{d-1}}\prod_{k=1}^{d-1}\Bigg[\Psi_{c_k}(x_k-t_k)+B_kh_k^2\phi_{c_k}^{(4)}(\eta_k-t_k)\Bigg]d\bold t_{d-1}\Bigg|\\
 \leq& (2C_0+\bar{C}^*\bar{B}_1)(c_d^2+h_d^2)|\ln c_d|\Bigg(1+C_1\sum_{k=1}^{d-1}(c_k^2+h_k^2)|\ln c_k|\Bigg).
 \end{split}
 \end{equation}
 Finally, combining Inequality \eqref{firstpart} and Inequality \eqref{secondpart}, we have
 \begin{equation*}
 \begin{split}
 &\Bigg|\int_{\mathbb{T}^d}(f^{(\boldsymbol{\alpha})}(\bold x)-f^{(\boldsymbol{\alpha})}(\bold t))\prod_{k=1}^d\Bigg[\Psi_{c_k}(x_k-t_k)+B_kh_k^2\phi_{c_k}^{(4)}(\eta_k-t_k)\Bigg]d\bold t\Bigg|\\
 \leq& \Bigg|\int_{\mathbb{T}^d}\Bigg(f^{(\boldsymbol{\alpha})}(\bold x_{d-1},x_d)-f^{(\boldsymbol{\alpha})}(\bold t_{d-1},x_d)\Bigg)\prod_{k=1}^d\Bigg[\Psi_{c_k}(x_k-t_k)+B_kh_k^2\phi_{c_k}^{(4)}(\eta_k-t_k)\Bigg]d\bold t_{d-1}dt_d\Bigg|\\
 &+\Bigg|\int_{\mathbb{T}^d}\Bigg(f^{(\boldsymbol{\alpha})}(\bold t_{d-1},x_d)-f^{(\boldsymbol{\alpha})}(\bold t_{d-1},t_d)\Bigg)\prod_{k=1}^d\Bigg[\Psi_{c_k}(x_k-t_k)+B_kh_k^2\phi_{c_k}^{(4)}(\eta_k-t_k)\Bigg]dt_dd\bold t_{d-1}\Bigg| \\
 \leq& \Bigg(\bar{C}\sum_{k=1}^{d-1}(c_k^2+h_k^2)|\ln c_k|\Bigg)\Bigg(1+C_1(c_d^2+h_d^2)|\ln c_d|\Bigg)+(2C_0+\bar{C}^*\bar{B}_1)(c_d^2+h_d^2)|\ln c_d|\Bigg(1+C_1\sum_{k=1}^{d-1}(c_k^2+h_k^2)|\ln c_k|\Bigg)\\
 \leq& C\sum_{k=1}^{d}(c_k^2+h_k^2)|\ln c_k|,~~\text{where}~C=2\max\{\bar{C},2C_0+\bar{C}^*\bar{B}_1\}.\\
 \end{split}
 \end{equation*}
That is, Inequality \eqref{induction} holds true for any $d$.
 \end{proof}
The above theorem provides a theoretical foundation of employing the convolution operator $\bold{C}_{\Psi_{\bold c,\bold h}}(f)$ and its derivatives to approximate both the function $f$ and its derivatives. However, in practical applications, we can only have  discrete function values $f|_{\bold X}$ evaluated at corresponding sampling centers $\bold X$. Thus we  have to adopt a discrete version of $\bold{C}_{\Psi_{\bold c,\bold h}}(f)$ as: \begin{equation}\label{general}
Q_{\bold X}f(\bold x)=\sum_{\mathbf{x_{\mathbf{i},\mathbf{j}}}\in \bold X}f(\mathbf{x_{\mathbf{i},\mathbf{j}}})\Psi_{\mathbf{c},\mathbf{h}}(\bold x-\mathbf{x_{\mathbf{i},\mathbf{j}}})\mathbf{\mu_{\mathbf{i},\mathbf{j}}},\ \bold x \in \mathbb{T}^d,
\end{equation} where $\{\mathbf{\mu_{\mathbf{i},\mathbf{j}}}\}$ denote  quadrature weights such that $Q_{\bold X}f(\bold x)$ is a  quadrature rule of the convolution operator $\bold{C}_{\Psi_{\mathbf{c},\mathbf{h}}}(f)(\bold x)$ for any fixed point $\bold x\in \mathbb{T}^d$. Further, we  can  decompose the approximation error $Q_{\bold X}f(\bold x)-f(\bold x)$ as a sum of two parts:
$$Q_{\bold X}f(\bold x)-f(\bold x)=\bold{C}_{\Psi_{\mathbf{c},\mathbf{h}}}(f)(\bold x)-f(\bold x)+Q_{\bold X}f(\bold x)-\bold{C}_{\Psi_{\mathbf{c},\mathbf{h}}}(f)(\bold x).$$ The first part is the convolution error of employing the convolution operator $\bold{C}_{\Psi_{\mathbf{c},\mathbf{h}}}$ to approximate $f$, while the second part is  the discretization error of using quadrature rule $Q_{\bold X}f$ to discretize  the convolution integral $\bold{C}_{\Psi_{\mathbf{c},\mathbf{h}}}(f)$.   Moreover, observing that the convolution error has been derived in Theorem \ref{convolutionerror}, we only need to derive the bound of discretization error $Q_{\bold X}f(\bold x)-\bold{C}_{\Psi_{\mathbf{c},\mathbf{h}}}(f)(\bold x)$. We remind readers that the discretization error depends heavily on quadrature rules \cite{Sloan}. Different quadrature rules lead to different quasi-interpolation schemes and corresponding error estimates, for example, Monte Carlo quasi-interpolation \cite{GaoandWu5}, Quasi-Monte Carlo quasi-interpolation \cite{GaoandWu6}. In the followings,  we shall take  the sparse grid  quadrature rule leading to  sparse grid quasi-interpolation.

 We adopt  notations provided in Subsection $2.4$. 
 Let $\{\mathbf{t_{\mathbf{l},\mathbf{j}}}=(t_{l_{1},j_{1}},t_{l_{2},j_{2}},\cdots ,t_{l_{d},j_{d}})\}$ be grid points of the sparse grid $\mathbb{W}_{n,d}$ over $\mathbb{T}^d$, where $t_{l_{k},j_{k}}=j_{k}2^{-l_{k}}$ for $j_{k}\in \{0,1,\cdots ,2^{l_{k}}\}$. Suppose that we have discrete function values  $\{f(\mathbf{t_{\mathbf{l},\mathbf{j}}})\}$ evaluated at   $\{\mathbf{t_{\mathbf{l},\mathbf{j}}}\}$. 
 Then, by discretizing the convolution integral $\bold{C}_{\Psi_{\bold c, \bold h}}(f)$ on the sparse grid, we get a sparse grid quasi-interpolant  
\begin{equation}\label{sparsequasi}
Q_{n,d}f(\bold x)=\sum_{\mathbf{t_{\mathbf{l},\mathbf{j}}}\in W_{n,d}}f(\mathbf{t_{\mathbf{l},\mathbf{j}}})\Psi_{\mathbf{c},\mathbf{h}}^*(\bold x-\mathbf{t_{\mathbf{l},\mathbf{j}}}),
\end{equation}
where $$\Psi_{\mathbf{c},\mathbf{h}}^*(\bold x)=\prod_{k=1}^d\Psi_{c_k,h_k}(x_k)\sin 2\pi h_k,\ h_k=2^{-l_{k}}.$$
It is straightforward to see that  Equation \eqref{sparsequasi} is a special case of  Equation \eqref{general} with sampling centers $\bold X=\mathbb{W}_{n,d}$ and quadrature weights $$\bold \mu_{\mathbf{l},\mathbf{j}}=\prod_{k=1}^d\sin 2\pi h_k.$$
 Moreover, due to the algebraic sum property of sparse grid $\mathbb{W}_{n,d}$, $Q_{n,d}f$ can be further represented as an algebraic sum of subgrid quasi-interpolants (Jeong, Kersey, and Yoon \cite{Jeong}, Usta and Levesley \cite{Usta}). More precisely, we have
\begin{equation}\label{sparsequasitele}
Q_{n,d}f(\bold x)=(-1)^{d-1}\sum_{k=0}^{d-1}(-1)^{k}\binom{d-1}{k}\sum_{|\mathbf{l}|=n+k}Q_{\mathbf{l}}f(\bold x)
\end{equation}
with
\begin{equation}\label{subgridquasi}
Q_{\mathbf{l}}f(\bold x)=\sum_{\mathbf{t_{\mathbf{l},\mathbf{j}}}\in \mathbb{W}_{\mathbf{l}}}f(\mathbf{t_{\mathbf{l},\mathbf{j}}})\Psi_{\mathbf{c},\mathbf{h}}^*(\bold x-\mathbf{t_{\mathbf{l},\mathbf{j}}}).
\end{equation}
 This may be the reason why sparse grid related approximations are also called as discrete blending or Boolean approximations \cite{BungartzandGriebel}.
In particular, for $d=2, 3$, we have \cite{Jeong}
$$Q_{n,2}f(\mathbf{x})=\sum_{\left |\mathbf{l}  \right |=n+1}Q_{\mathbf{l}}f(\mathbf{x})-\sum_{\left |\mathbf{l}  \right |=n}Q_{\mathbf{l}}f(\mathbf{x}),$$
$$Q_{n,3}f(\mathbf{x})=\sum_{\left |\mathbf{l}  \right |=n+2}Q_{\mathbf{l}}f(\mathbf{x})-2\sum_{\left |\mathbf{l}  \right |=n+1}Q_{\mathbf{l}}f(\mathbf{x})+\sum_{\left |\mathbf{l}  \right |=n}Q_{\mathbf{l}}f(\mathbf{x}).$$

  We now derive simultaneous error estimates of  $Q_{n,d}f$.
Based on  above discussions, we only need to  bound the discretization error $D^{\boldsymbol{\alpha}}Q_{n,d}f(\bold x)-D^{\boldsymbol{\alpha}}\bold C_{\Psi_{\bold c, \bold h}}(f)(\bold x)$ for each fixed point $\bold x\in \mathbb{T}^d$. To this end, we first represent it as
\begin{equation}\label{sparsecase}
\begin{split}
D^{\boldsymbol{\alpha}}Q_{n,d}f(\bold x)-D^{\boldsymbol{\alpha}}\bold C_{\Psi_{\bold c, \bold h}}(f)(\bold x)&=D^{\boldsymbol{\alpha}}Q_{n,d}f(\bold x)-\bold C_{D^{\boldsymbol{\alpha}}\Psi_{\bold c, \bold h}}(f)(\bold x)\\
&=\sum_{\mathbf{t_{\mathbf{l},\mathbf{j}}}\in \mathbb{W}_{n,d}}f(\mathbf{t_{\mathbf{l},\mathbf{j}}})D^{\boldsymbol{\alpha}}\Psi^*_{\bold c, \bold h}(\bold x-\mathbf{t_{\mathbf{l},\mathbf{j}}})-\int_{\mathbb{T}^d}f(\bold t)D^{\boldsymbol{\alpha}}\Psi_{\bold c, \bold h}(\bold x-\bold t)d\bold t\\
&=(-1)^{d-1}\sum_{l=0}^{d-1}(-1)^{l}\binom{d-1}{l}\sum_{|\mathbf{l}|=n+l}\Bigg(D^{\boldsymbol{\alpha}}Q_{\mathbf{l}}f(\bold x)-\int_{\mathbb{T}^d}f(\bold t)D^{\boldsymbol{\alpha}}\Psi_{\bold c, \bold h}(\bold x-\bold t)d\bold t\Bigg)
\end{split}
\end{equation}
using  the algebraic sum property \eqref{sparsequasitele} of $Q_{n,d}f$ and the identity $$(-1)^{d-1}\sum_{l=0}^{d-1}(-1)^{l}\binom{d-1}{l}\sum_{|\mathbf{l}|=n+l}1\equiv1.$$ Then we go further with deriving the bound of  $$D^{\boldsymbol{\alpha}}Q_{\bold l}f(\bold x)-\int_{\mathbb{T}^d}f(\bold t)D^{\boldsymbol{\alpha}}\Psi_{\bold c, \bold h}(\bold x-\bold t)d\bold t,$$ which is an error of tensor product  quadrature formula with a directionally uniform stepsize $\bold h=(2^{-l_1},\cdots, 2^{-l_d})$.  There are many techniques to derive  error bounds of multivariate tensor product problems, see  Gerstner and Griebel \cite{Gerstner}, Wasilkowski and Wo$\acute{z}$niakowski \cite{Wozniakowski}, for instance. In particular,  for any periodic function $f\in W_{\infty}^{r+2}(\mathbb{T}^d)$, we have
\begin{equation}\label{singlelevel}
\begin{split}
&\Bigg|D^{\boldsymbol{\alpha}}Q_{\bold l}f(\bold x)-\int_{\mathbb{T}^d}f(\bold t)D^{\boldsymbol{\alpha}}\Psi_{\bold c, \bold h}(\bold x-\bold t)d\bold t\Bigg|
\leq
C\sum_{j=1}^dh_j^{r+2}\Bigg|\int_{\mathbb{T}}D^{r+2+\alpha_j}\Psi_{c_j, h_j}(x_j-t_j)dt_j\Bigg|\\
\leq& 4\pi^2\frac{1-\cos(\pi h_j)}{\sin(2\pi h_j) \sin(\pi h_j)}C\sum_{j=1}^dh_j^{r+2}\Bigg|\int_{\mathbb{T}}\phi_{c_j}^{(r+2+\alpha_j)}(x_j-t_j)dt_j\Bigg|\\
\leq& C^*\sum_{j=1}^dh_j^{r+2}\left(\int_{|\sin\pi(\eta_j-t_j)|\leq c_j}c_j^{-r-1-\alpha_j}dt_j+\int_{|\sin\pi(\eta_j-t_j)|\geq c_j}\frac{1}{|\sin\pi(\eta_j-t_j)|^{r+1+\alpha_j}}dt_j\right)\\
\leq& C^*\sum_{j=1}^dh_j^{r+2}\left(c_j^{-r-\alpha_j}+c_j^{-r-\alpha_j}\int_{|\sin\pi(\eta_j-t_j)|\geq c_j}\frac{1}{|\sin\pi(\eta_j-t_j)|}dt_j\right)\leq \bar{C}^*\sum_{j=1}^dh_j^{r+2}c_j^{-r-\alpha_j}|\ln c_j|,
\end{split}
\end{equation}based on Inequalities \eqref{eq:lem_highorderofphi3}-\eqref{bound}. 
Finally, combining Equation \eqref{sparsecase} and Inequality \eqref{singlelevel}, we can obtain the following error bound:
\begin{equation*}
\begin{split}
&\Bigg|\sum_{\mathbf{t_{\mathbf{l},\mathbf{j}}}\in \mathbb{W}_{n,d}}f(\mathbf{t_{\mathbf{l},\mathbf{j}}})D^{\boldsymbol{\alpha}}\Psi_{\bold c, \bold h}^*(\bold x-\mathbf{t_{\mathbf{l},\mathbf{j}}})-\int_{\mathbb{T}^d}f(\bold t)D^{\boldsymbol{\alpha}}\Psi_{\bold c, \bold h}(\bold x-\bold t)d\bold t\Bigg|\\
\leq& \sum_{l=0}^{d-1}\binom{d-1}{l}\sum_{|\mathbf{l}|=n+l}\Bigg|D^{\boldsymbol{\alpha}}Q_{\mathbf{l}}f(\bold x)-\int_{\mathbb{T}^d}f(\bold t)D^{\boldsymbol{\alpha}}\Psi_{\bold c, \bold h}(\bold x-\bold t)d\bold t\Bigg|\\
\leq& \bar{C}^*\sum_{l=0}^{d-1}\binom{d-1}{l}\sum_{|\mathbf{l}|=n+l}\sum_{j=1}^dh_j^{r+2}c_j^{-r-\alpha_j}|\ln c_j|\leq C_3\sum_{j=1}^dh_j^{r+2}c_j^{-r-\alpha_j}|\ln c_j|.
\end{split}
\end{equation*}
We conclude the above discussions into the following theorem.
\begin{theorem}\label{discretizationerror1}
Let $Q_{n,d}f$ and $C_{\Psi_{\bold c, \bold h}}(f)$ be defined as above. Then, for any $0\leq \|\boldsymbol{\alpha}\|_{\infty}\leq r$ and any $f\in W_{\infty}^{r+2}(\mathbb{T}^d)$, we have the error estimate
$$\|D^{\boldsymbol{\alpha}}Q_{n,d}f-D^{\boldsymbol{\alpha}}C_{\Psi_{\bold c, \bold h}}(f)\|_{\infty}\leq C_{3}\sum_{j=1}^dh_j^{r+2}c_j^{-r-\alpha_j}|\ln c_j|.$$
\end{theorem}
Consequently,  combining Theorem \ref{convolutionerror} and Theorem \ref{discretizationerror1}, we can derive  simultaneous approximation orders of $Q_{n,d}f$ to $f$ as follows.
\begin{theorem}\label{maintheorem}
Let $Q_{n,d}f$ be the sparse grid quasi-interpolant defined as in Equation \eqref{sparsequasi}. Then for any $0\leq \|\boldsymbol{\alpha}\|_{\infty}\leq r$ and any $f\in W_{\infty}^{r+2}(\mathbb{T}^d)$, we have
\begin{equation}\label{finaleror}
\|D^{\boldsymbol{\alpha}}Q_{n,d}f-D^{\boldsymbol{\alpha}}f\|_{\infty}=\mathcal{O}\Bigg(\sum_{j=1}^d(c_j^2+h_j^2)|\ln c_j|\Bigg)+\mathcal{O}\Bigg(\sum_{j=1}^dh_j^{r+2}c_j^{-r-\alpha_j}|\ln c_j|\Bigg).
\end{equation}
Moreover, we can choose directionally optimal shape parameters $c_j=\mathcal{O}(h_j^{(r+2)/(r+2+\alpha_j)})$, $j=1,2,\cdots,d$, such that Equation \eqref{finaleror} becomes
\begin{equation}\label{finalerorop}
\|D^{\boldsymbol{\alpha}}Q_{n,d}f-D^{\boldsymbol{\alpha}}f\|_{\infty}=\mathcal{O}(\sum_{j=1}^dh_j^{(2r+4)/(r+2+\|\boldsymbol{\alpha}\|_{\infty})}|\ln h_j|).
\end{equation}
\end{theorem}

Note that our sparse grid quasi-interpolant $Q_{n,d}f$ is a tensor-product of univariate MQ trigonometric B-spline quasi-interpolant \cite{GaoandWu2} that preserves trigonometric convexity \cite{Koch} of target function. We can derive a corresponding shape-preserving property of $Q_{n,d}f$  as follows. 
\begin{lemma}\label{shapepreser}
If $f$ is a nonnegative periodic function, then $Q_{n,d}f$ is also a nonnegative periodic function. Moreover, if $f$ is a   trigonometric convex function with respect to  coordinate $x_k$ (that is $(D_{x_k}^2+\pi^2I)f(\bold x)\geq 0$), then  $Q_{n,d}f$ is also a trigonometric convex function with respect to $x_k$.
\end{lemma}
\begin{proof}
Observe that  
$$\Psi_{\mathbf{c},\mathbf{h}}^*(\bold x)=\prod_{k=1}^d(D_{x_k}^2+\pi^2I)\phi_{c_k}(x_k-\xi_k)\sin 2\pi h_k, \ \xi_k\in (-h_k, h_k).$$ The nonnegativity-preserving property follows directly from $$(D_{x_k}^2+\pi^2I)\phi_{c_k}(x_k-\xi_k)=\frac{\pi c_k^2(c_k^2+1)}{2(c_k^2+\sin^2\pi (x_k-\xi_k))^{3/2}}>0,\ \text{for all} \ k=1,2,\cdots, d.$$  To prove that $Q_{n,d}f$ preserves the trigonometric convexity \cite{Koch}, we apply the technique proposed in \cite{GaoandWu1} to $x_k$ and rewrite $Q_{n,d}f(\bold x)$ as
\begin{equation*}
Q_{n,d}f(\bold x)=\sum_{\mathbf{t_{\mathbf{l},\mathbf{j}}}\in W_{n,d}}\Bigg(\prod_{i\neq k}^d[-h_i,0,h_i]_{Tr_2}\phi_{c_i}(x_i-t_{l_i,j_i})\sin2\pi h_i\Bigg)[-h_k,0,h_k]_{Tr_2}f(\mathbf{t_{\mathbf{l},\mathbf{j}}})\phi_{c_k}(x_k-t_{l_k,j_k})\sin2\pi h_k,
\end{equation*}
where $h_j=2^{-l_j}$ for $j=1,2,\cdots, d$.
This in turn leads to \begin{equation*}
\begin{split}
&(D_{x_k}^2+\pi^2 I)Q_{n,d}f(\bold x)\\
&=\sum_{\mathbf{t_{\mathbf{l},\mathbf{j}}}\in W_{n,d}}\Bigg(\prod_{i\neq k}^d[-h_i,0,h_i]_{Tr_2}\phi_{c_i}(x_i-t_{l_i,j_i})\sin2\pi h_i\Bigg)[-h_k,0,h_k]_{Tr_2}f(\mathbf{t_{\mathbf{l},\mathbf{j}}})(D_{x_k}^2+\pi^2 I)\phi_{c_k}(x_k-t_{l_k,j_k})\sin2\pi h_k.
\end{split}
\end{equation*}
Additionally, the trigonometric convexity of $f$ with respect to $x_k$ gives $[-h_k,0,h_k]_{Tr_2}f\geq 0$. 
Consequently, by observing that $\prod_{i\neq k}^d[-h_i,0,h_i]_{Tr_2}\phi_{c_i}(x_i-t_{l_i,j_i})\sin2\pi h_i>0$ and $(D_{x_k}+\pi^2 I)\phi_{c_k}(x_k-t_{l_k,j_k})\sin2\pi h_k>0$, we have 
$(D_{x_k}^2+\pi^2 I)Q_{n,d}f(\bold x)\geq 0$.
\end{proof}

Up to now, we have constructed a sparse grid quasi-interpolant with anisotropic tensor-product multiquadric trigonometric kernel for high-dimensional periodic function approximation. We  go further with extending  our construction to quasi-interpolation for non-periodic function approximation  over the hypercube $[0,1]^d$.

\subsection{Quasi-interpolation defined on  $[0,1]^d$}
With the periodization strategy introduced in Subsection $2.2$, given any grid point $\mathbf{t_{\mathbf{l},\mathbf{j}}}=(t_{l_1,j_1},t_{l_2,j_2},\cdots,t_{l_d,j_d})$ over the sparse grid $\mathbb{W}_{n,d}$ of $\mathbb{T}^d$,  we  first transform it to the hypercube $[0,1]^d$  as $\boldsymbol{\gamma}(\mathbf{t_{\mathbf{l},\mathbf{j}}}):=(\gamma_1(t_{l_1,j_1}),\cdots,\gamma_d(t_{l_d,j_d}))\in [0,1]^d$ using the torus-to-cube transformation $\boldsymbol{\gamma}$. Suppose that we have   discrete function values of an unknown  \textbf{non-periodic} function $g$ as $$\{g(\boldsymbol{\gamma}(\mathbf{t_{\mathbf{l},\mathbf{j}}})):=g(\gamma_1(t_{l_1,j_1}),\cdots,\gamma_d(t_{l_d,j_d})), \mathbf{t_{\mathbf{l},\mathbf{j}}}\in \mathbb{W}_{n,d}\}.$$ Then, based on the torus-to-cube transformation $\boldsymbol{\gamma}$, we can obtain corresponding discrete function values of the transformed \textbf{periodic} function $f$
evaluated at the sparse grid $\mathbb{W}_{n,d}$ as
$$f(\mathbf{t_{\mathbf{l},\mathbf{j}}}):=g(\gamma_1(t_{l_1,j_1}),\cdots,\gamma_d(t_{l_d,j_d}))\prod_{i=1}^d[\omega_i(\gamma_i(t_{l_i,j_i}))\gamma_i'(t_{l_i,j_i})]^{1/2}, \mathbf{t_{\mathbf{l},\mathbf{j}}}\in \mathbb{W}_{n,d}.$$ Further,   applying the quasi-interpolation scheme \eqref{sparsequasi} to the discrete function values $\{f(\mathbf{t_{\mathbf{l},\mathbf{j}}})\}$, we get a sparse grid quasi-interpolant
\begin{equation*}
Q_{n,d}f(\bold x)=\sum_{\mathbf{t_{\mathbf{l},\mathbf{j}}}\in \mathbb{W}_{n,d}}f(\mathbf{t_{\mathbf{l},\mathbf{j}}})\Psi_{\mathbf{c},\mathbf{h}}^*(\bold x-\mathbf{t_{\mathbf{l},\mathbf{j}}})
\end{equation*}
 of the transformed \textbf{periodic} function $f$ defined over $\mathbb{T}^d$. To get a corresponding   quasi-interpolant $Qg(\bold y)$ of $g$ defined over $[0, 1]^d$ from $Q_{n,d}f(\bold x)$, we   first introduce a non-negative $L_1$-integrable  density function $\boldsymbol{\varrho}$ of $\boldsymbol{\gamma}$  via $$\boldsymbol{\varrho}(\bold y)=\prod_{j=1}^d\varrho_j(y_j), \  \text{where}\ \varrho_j(y_j):=(\gamma_j^{-1})'(y_j)=\frac{1}{\gamma_j'(\gamma_j^{-1}(y_j))}.$$ Here the inverse transformation $\boldsymbol{\gamma}^{-1}=(\gamma_1^{-1}(y_1),\cdots,\gamma_d^{-1}(y_d)):[0,1]^d\rightarrow \mathbb{T}^d$ is defined in the sense of $\bold x=\boldsymbol{\gamma}^{-1}(\bold y)\in \mathbb{T}^d$ if and only if $\bold y=\boldsymbol{\gamma}(\bold x)\in [0,1]^d$. Then, with the help of the density function $\boldsymbol{\varrho}$, we construct our final quasi-interpolation $Qg(\bold y),\ \bold y \in [0,1]^d$, in the form
\begin{equation}\label{nonperiodic}
Qg(\bold y)=\sum_{\mathbf{t_{\mathbf{l},\mathbf{j}}}\in W_{n,d}}g(\gamma_1(t_{l_1,j_1}),\cdots,\gamma_d(t_{l_d,j_d}))\prod_{i=1}^d\Bigg(\sqrt{\omega_i(\gamma_i(t_{l_i,j_i}))\gamma_i'(t_{l_i,j_i})}
\sqrt{\frac{\rho_i(y_i)}{\omega_i(y_i)}}\Psi_{c_i,h_i}(\gamma_i^{-1}(y_i)-t_{l_i,j_i})\sin2\pi h_i\Bigg).
\end{equation}
Moreover, corresponding to  the $L_{\infty}$-norm approximation error of $Q_{n,d}f$ (that is Theorem \ref{maintheorem}), we can  derive a weighted $L_{\infty}$-norm approximation error of $Qg$ in the following theorem.
\begin{theorem}(Weighted  $L_{\infty}$-norm approximation error)\label{errorestimateofnoninfty}
Let $Qg$ be defined as above for any $g\in L_2([0,1]^d,\omega)\cap C_{\text{mix}}^l([0,1]^d)$. Let $\boldsymbol{\gamma}$ be the torus-to-cube transformation satisfying conditions in Lemma \ref{lemmaH(T)}. Then we have
\begin{equation}\label{errorinfinity}
\begin{split}
\|Qg-g\|_{L_{\infty}([0,1]^d,\sqrt{\frac{\boldsymbol{\omega}}{\boldsymbol{\rho}}})}&=\|Q_{n,d}f-f\|_{\infty}.\\
\end{split}
\end{equation}
\end{theorem}
\begin{proof}
The conditions of $g$ and $\boldsymbol{\gamma}$ lead to a transformed periodic function $f\in \mathbb{H}^l(\mathbb{T}^d)\subset C(\mathbb{T}^d)$. Thus we have
\begin{equation*}
\begin{split}
\|Qg-g\|_{L_{\infty}([0,1]^d,\sqrt{\frac{\boldsymbol{\omega}}{\boldsymbol{\rho}}})}&=\text{esssup}_{\bold y\in [0, 1]^d}\Bigg|\sqrt{\frac{\boldsymbol{\omega}}{\boldsymbol{\rho}}}\Bigg(g(\bold y)-Qg(\bold y)\Bigg)\Bigg|\\
&=\text{esssup}_{\bold x\in \mathbb{T}^d}\Bigg|g(\boldsymbol{\gamma}(\bold x))\prod_{j=1}^d\sqrt{\omega_j(\psi_j(x_j))\gamma_j'(x_j)}-Q_{n,d}f(\bold x)\Bigg|\\
&=\|f-Q_{n,d}f\|_{\infty}.
\end{split}
\end{equation*}
\end{proof}
This theorem implies that, in terms of the weighted $L_{\infty}$-norm approximation error,  the quasi-interpolant $Qg$ provides the same approximation order as its periodic counterpart $Q_{n,d}f$.
\section {\textit{Numerical Simulations}}
This section consists of two subsections. The first subsection demonstrates examples of approximating periodic functions and their  derivatives using  our sparse grid quasi-interpolation scheme \eqref{sparsequasi}. We also compare our scheme with the one based on Fourier techniques proposed by Nasdala and Potts \cite{NasdalaandPotts}. The second subsection provides an example of approximating a non-periodic function with our extended quasi-interpolation \eqref{nonperiodic} and the kernel-based quasi-interpolation scheme discussed by Jeong, Kersey and Yoon \cite{Jeong}.
\subsection{\textit{Periodic function approximation over $\mathbb{T}^d$}}
As an example, we consider approximating the purpose-built periodic polynomials having  a finite order of differentiability and periodicity in reference \cite{Morrow}. Following the technique provided by Morrow and  Stoyanov \cite{Morrow}, we construct a  multivariate periodic polynomial $f\in W^{4}_{\infty}(\mathbb{T}^{d})$  in a tensor-product form
\begin{equation}\label{periodipolynomial}
f_d(\bold x)=\prod_{j=1}^dg(x_j), \ \bold x\in \mathbb{T}^{d},
\end{equation}
where $$g(x_j)=\frac{(2x_{j}-1)^6}{5}-\frac{(2x_{j}-1)^4}{1}+\frac{7(2x_{j}-1)^2}{5}\in W^{4}_{\infty}(\mathbb{T}),\ j=1,2,\cdots, d.$$

We first  compare our scheme with the Fourier-based approximation scheme \cite{NasdalaandPotts} for approximating $f_5$ on sparse grid. Numerical results under different sparse grid levels $n$ are provided in Table $1$, where the $L_{\infty}$-norm   approximation error are evaluated over $3125$ fully gridded prediction points in $[0, 1]^5$.  Following Theorem $3.1$, we choose the shape parameters $c_{j}=A_{j}h_{j}$ with  $\{A_j\}$ being provided in Table $3$ in the Appendix  by trials and errors.
We can observe that the scheme \cite{NasdalaandPotts} provides higher convergence rate, but it takes much more running time, which makes it  difficult   to  approximate high-dimensional (greater than $5$) function even on sparse grid.  

\begin{table}[!http]
\caption{ Numerical results of approximating $f_5\in W_{\infty}^{4}([0,1]^{5})$ on sparse grid}
\begin{tabular}{ccccccc}
\hline
\multirow{2}{*}{$n$} & \multicolumn{2}{c}{Approx error} & \multicolumn{2}{c}{Approx order} & \multicolumn{2}{c}{Running time(/s)} \\ \cline{2-7} 
                   & Our scheme   & Fourier scheme  & Our scheme    & Fourier scheme    & Our scheme    & Fourier scheme   \\ \hline
5                  & 2.9597e-04      & 0.0032      & -            & -                 & 0.72         & 63.16             \\
6                  & 7.3523e-05      & 2.0143e-04      & 2.0242       & 3.9897            & 0.80         & 226.92             \\
7                  & 1.8243e-05      & 6.0373e-06      & 2.0395       & 5.0602           & 2.64         & 930.98             \\
8                  & 4.3529e-06  & 2.4039e-07      & 2.0043       & 4.6505             & 6.68         & 3527.78            \\ \hline
\end{tabular}
\centering
\end{table}

Moreover, to demonstrate that our sparse grid quasi-interpolation can lessen the curse of dimensionality, we go further with employing it to approximate $f_d$ and their derivatives for   $d\in \{7,\ 10\}$. We approximate the function, its first-order derivative  $\partial f_d/\partial x_1$, and second-order derivative $\partial^2 f_d/\partial x_1^2$. We choose  shape parameters $c_{j}=A_{j}h_{j}^{4/(4+\alpha _{j})},\ j=1,\ 2$ with $\alpha_{j} \in \{0,\ 1,\ 2\}$  following Theorem $3.1$. Corresponding constant coefficients $\{A_j\}$ are provided in Tables $4$-$5$ in the Appendix  by trials and errors. We note that   coefficients for approximating partial derivatives are usually larger than the ones for approximating the function.  A posteriori and a priori approximation orders  under levels (of sparse grid) $n\in \{3,\ 4,\ \cdots,\ 8\}$ for $d=7$,  and $n\in \{2,\ \cdots, 6\}$ for $d=10$, are provided in Figure $2$. Here the $L_{\infty}$-approximation errors are evaluated over $2187,\ 1024$ gridded prediction points corresponding to $d=7,10$.


\begin{figure}[!htp]
\centering
\subfigure[ $d=7$]{
\begin{minipage}[t]{0.5\linewidth}
\centering
\includegraphics[scale=0.5]{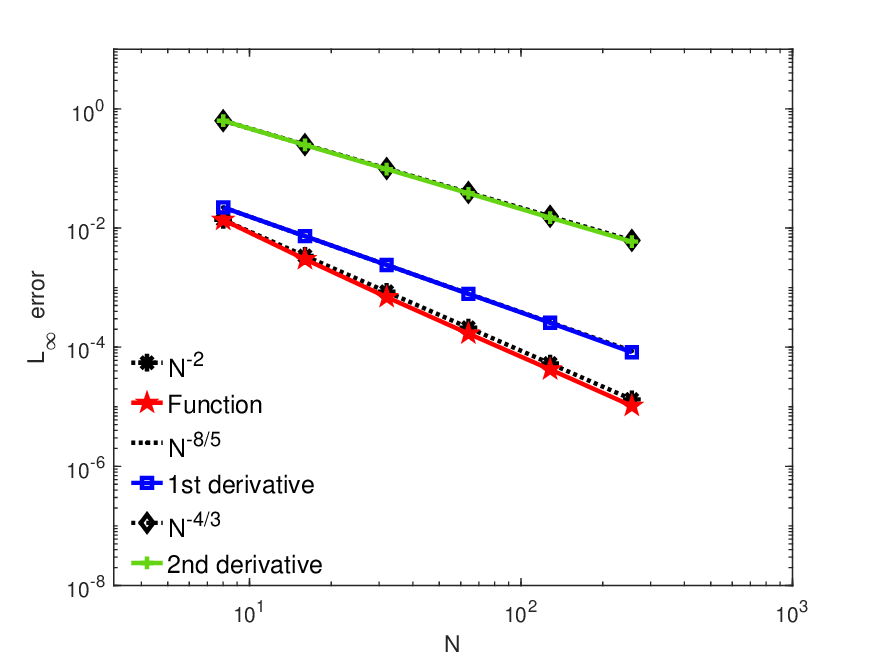}
\end{minipage}%
}%
\subfigure[ $d=10$]{
\begin{minipage}[t]{0.5\linewidth}
\centering
\includegraphics[scale=0.5]{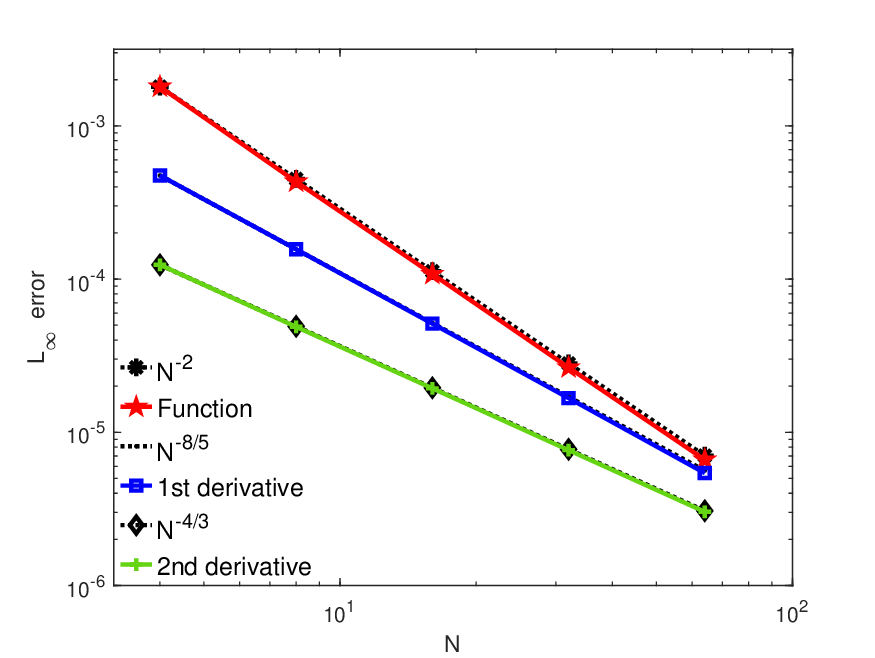}
\end{minipage}%
}%
\caption{ Numerical results of approximating derivatives of $f_d\in W^{4}_{\infty}([0,1]^{d})$ on sparse grid}
\end{figure}



 In  Figure $2$, the asterisk, point and diamond dotted lines denote a priori approximation order $2$ for $||\boldsymbol{\alpha}||_{\infty}=0$,  a priori approximation order $8/5$ for $||\boldsymbol{\alpha}||_{\infty}=1$, and a priori approximation order $4/3$ for $||\boldsymbol{\alpha}||_{\infty}=2$, respectively. Besides, the red pentagram, blue square and green plus-sign lines denote corresponding posteriori approximation orders of approximating the function, first-order derivative and second-order derivatives with our scheme, respectively. From Figure $2$, we can find that all posteriori approximation orders of our scheme are  higher than their priori counterparts. These demonstrate  vividly that our sparse grid quasi-interpolation still performs well for approximating high-dimensional function and its derivatives and thus can lessen the curse of dimensionality.

\subsection{\textit{Non-periodic function approximation over unit hypercube}}
We go further with providing an example of employing our quasi-interpolation \eqref{nonperiodic} to approximate a non-periodic function. As an example, we take the test function  (given in reference \cite{NasdalaandPotts}):
\begin{equation}\label{nonfunction}
g(\mathbf{y})=\prod_{j=1}^{d}(y_{j}^{2}-y_{j}+\frac{3}{4}),\  \mathbf{y}\in [-0.5,0.5]^{d}.
\end{equation}
Moreover, we choose the weight function $\boldsymbol{\omega}=1$ and the logarithmic transformation $\gamma_j(x_j,\eta_j)$ as
\begin{equation}\label{e1}
\begin{split}
\gamma_j(x_j,\eta_j)&=\frac{1}{2}\frac{(1+2x_j)^{\eta_j}-(1-2x_j)^{\eta_j}}{(1+2x_j)^{\eta_j}+(1-2x_j)^{\eta_j}},\\
\gamma_j'(x_j,\eta_j)&=\frac{4\eta_j(1-4x_j^2)^{\eta_j-1}}{((1+2x_j)^{\eta_j}+(1-2x_j)^{\eta_j})^2},
\end{split}
\end{equation}
 with $x_j\in [-0.5,0.5]$ and $\eta_j \in \mathbb{R}_{+}$. The non-periodic function together with the logarithmic transformation leads to a transformed periodic function
\[
\begin{split}\label{logtransfunction}
f(\mathbf{x})&=g(\gamma_{1} (x_{1},\eta_{1}),\gamma_{2} (x_{2},\eta_{2}),\cdots ,\gamma_{d} (x_{d},\eta_{d}))\prod_{j=1}^{d}\sqrt{{\gamma }'_{j}(x_{j},\eta_{j})}\\
&=\prod_{j=1}^{d}((\frac{1}{2}\frac{(1+2x_{j})^{\eta _{j}}-(1-2x_{j})^{\eta _{j}}}{(1+2x_{j})^{\eta _{j}}+(1-2x_{j})^{\eta _{j}}})^{2}-(\frac{1}{2}\frac{(1+2x_{j})^{\eta _{j}}-(1-2x_{j})^{\eta _{j}}}{(1+2x_{j})^{\eta _{j}}+(1-2x_{j})^{\eta _{j}}})+\frac{3}{4})\prod_{j=1}^{d}\sqrt{\frac{4\eta _{j}(1-4x_{j}^{2})^{\eta_j -1}}{((1+2x_{j})^{\eta _{j}}+(1-2x_{j})^{\eta _{j}})^{2}}}.
\end{split}
\]
In general, it is difficult to construct a suitable cube-to-torus map for a given weight function, for more details on  constructing a map corresponding to a given weight function, we refer readers to \cite{NasdalaandPotts}.
We discretize  the weighted $L_{\infty}$-norm on $[-0.5,0.5]^{d}$ as
$$||g-Qg||_{L_{\infty}([-0.5,0.5]^{d},\sqrt{1/\boldsymbol{\rho}})}=\max_{i\in\{1,\cdots,N\}}\Bigg\{\Bigg|\sqrt{1/\boldsymbol{\rho}(\mathbf{y}_{i})}
(g(\mathbf{y}_{i})-Qg(\mathbf{y}_{i}))\Bigg|\Bigg\}.$$
Numerical results are provided in Table $2$ under the choice of the shape parameters $c_j=A_jh_j$, and the logarithmic transformation $\gamma_j$ with $\eta_j=4$, for $j=1, 2$. Here discrete weighted $L_{\infty}$-norm approximation errors are computed over $101\times101$ uniformly gridded prediction points in $[-0.5, 0.5]^2$. From the table, we can find that our sparse grid quasi-interpolation provides almost the same convergence rate as its full grid counterpart, while saving a lot of sampling data and running time.


\begin{table}[!http]
\caption{ Numerical results of approximating two-dimensional non-periodic function under $\eta_j =4$}
 \begin{tabular}{ccccccccc}
\hline
\multirow{2}{*}{$n$} & \multicolumn{2}{c}{Choices of $\{A_j\}$} & \multicolumn{2}{c}{Approx error} & \multicolumn{2}{c}{Approx order} & \multicolumn{2}{c}{Running time(/s)} \\ \cline{2-9}
                    & sparse                  & full                   & sparse        & full          & sparse          & full           & sparse         & full\\ \hline
 6                  & (2.7,2.7)             & (0.38,0.38)            & 0.1990        & 0.0173        & -               & -              & 0.05           & 0.16   \\
 7                  & (0.98,0.98)             & (0.38,0.38)                  & 0.0497        & 0.0040        & 2.0015          & 2.1127         & 0.09           & 0.59  \\
 8                  & (0.54,0.54)             & (0.30,0.30)            & 0.0124        & 0.0010        & 2.0029          & 2.0000         & 0.24           & 2.40    \\
9                  & (0.40,0.40)             & (0.30,0.30)            & 0.0031        & 2.3676e-04    & 2.0000          & 2.0785         & 0.54           & 10.63 \\
 10                 & (0.11,0.11)             & (0.02,0.02)            & 7.5896e-04        & 5.6487e-05    & 2.0302          & 2.0674         & 1.20           & 78.74   \\
 11                 & (0.01,0.01)             & (0.01,0.01)            & 1.8896e-04        & 1.3580e-05    & 2.0059 & 2.0564         & 2.61           & 439.65  \\ \hline
\end{tabular}
 \centering
 \end{table}

Further, to demonstrate that our scheme can mitigate the curse of dimensionality, we consider higher dimension cases for $d=7,10$. We also compare our scheme with the kernel-based quasi-interpolation scheme on sparse grids that has been recently proposed by Jeong, Kersey and Yoon \cite{Jeong} via the boundary extension technique.  Moreover, to enable a close comparison with our multiquadric trigonometric function, we choose the multiquadric function  $\phi_{1,c_j}(x_j)=\sqrt{c_j^2+x_j^2}$ as a special case of formulas $(6)$ in \cite{Jeong}.
 Then the kernels in formulas $(11)$ in \cite{Jeong} corresponds to the ones of $L_d$ quasi-interpolation proposed by Schaback and Wu \cite{WuandSchaback}.  Besides, it includes the second-order B-spline quasi-interpolation as a special case with all  shape parameters being zero.
Finally, using the tensor-product of these kernels as the final multivariate kernel (that is the formulas $(3)$ in  \cite{Jeong}), we arrive at the scheme $(4)$ in \cite{Jeong} with the simplest (two points) boundary extension technique. 
Sketches of corresponding approximation errors as well as running time are provided in Figure $3$.  Here the discrete weighted $L_{\infty}$-norm approximation errors are evaluated over $2187,1024$ gridded prediction points for $d=7,10$, espectively. For both schemes, we choose the same shape parameters $\bold c=\bold A\bold h$ with $\mathbf{A}$ being provided in Table $6$ in the Appendix by trials and errors. In the figure, the  black dashed line denotes a priori approximation order, the green solid lines denote the numercial results of the scheme proposed by Jeong, Kersey and Yoon \cite{Jeong} (abbreviated as JKY), the black dash-dotted lines denote the one of the second-order (tensor-product) B-spline quasi-interpolation, and the red square lines denote the numerical results of our scheme with  $\eta_j=2$ in the logarithmic transformation $\gamma_j(x_j,\eta_j)$ defined in Equation \eqref{e1}. We can find that our scheme yields the same convergence order as JKY and the tensor-product B-spline quasi-interpolation. However, it saves a lot of running time. More importantly, our scheme still performs well for $d=10$, while  JKY and the B-spline quasi-interpolation even with the simplest (two points) boundary extension can not run in the same computer. 
\begin{figure}

\subfigure[ $d=7$, $\eta_j=2$]{
\begin{minipage}[t]{0.99\linewidth}
\centering
\includegraphics[scale=0.5]{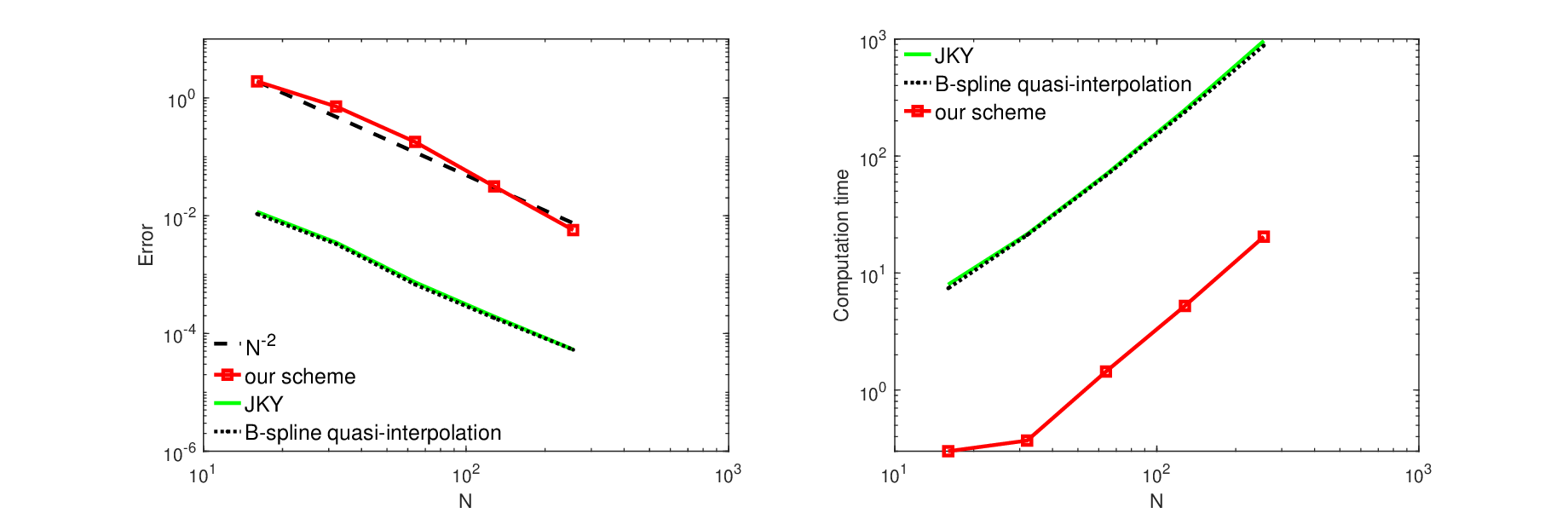}
\end{minipage}%
}%

\subfigure[ $d=10$, $\eta_j=2$]{
\begin{minipage}[t]{0.99\linewidth}
\centering
\includegraphics[scale=0.5]{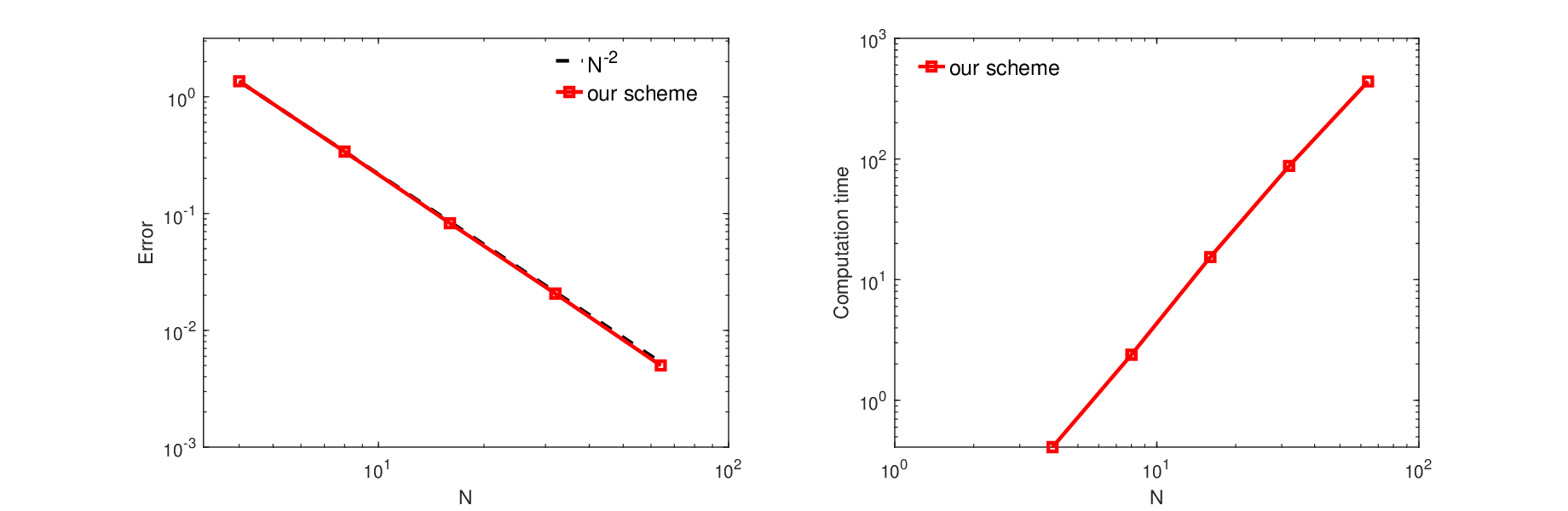}
\end{minipage}%
}%
\caption{ Numerical results of approximating high-dimensional non-periodic  function}
\end{figure}

\section {\textit{Conclusions and discussions}}
We  discuss high-dimensional function approximation under the framework of quasi-interpolation.  Our final approximant takes a weighted average of the available data and directly assembles them with translations of   tensor-product univariate kernels. It enjoys a simple construction, optimal convergence rate, as well as regularization property, yet it provides an efficient tool for mitigating the curse of dimensionality in various high-dimensional environments. Further works will focus on  developing   $L_p$-estimates  of our quasi-interpolation for more general function spaces motivated from the salient work given by Kolomoitsev et al., \cite{Kolomoitsev3} and applying it in numerical solutions of differential equations defined over  high-dimensional tori.\\
\textbf{Acknowledgements} \\
The authors are very grateful to the editor as well as two anonymous referees for  valuable suggestions and insightful comments, which contributed to improve the paper considerably.

\section*{Appendix}

\begin{table}[H]
\caption{ Choices of $\{A_j\}$ for $f\in W^{4}_{\infty}([0,1]^{5})$}
\begin{tabular}{clll}
\hline
$n$  & \multicolumn{1}{c}{Function} & \multicolumn{1}{c}{1st-order derivatives} &  \multicolumn{1}{c}{2nd-order derivatives} \\ \hline
5  & (0.29,0.29,$\cdots$,0.29)   & (0.04,0.01,$\cdots$,0.01)           & (0.19,0.01,$\cdots$,0.01)           \\
6  & (0.22,0.22,$\cdots$,0.22)   & (0.06,0.01,$\cdots$,0.01)           & (0.26,0.01,$\cdots$,0.01)           \\
7  & (0.18,0.18,$\cdots$,0.18)   & (0.08,0.01,$\cdots$,0.01)           & (0.31,0.01,$\cdots$,0.01)           \\
8  & (0.15,0.15,$\cdots$,0.15)   & (0.10,0.01,$\cdots$,0.01)           & (0.33,0.01,$\cdots$,0.01)           \\
9  & (0.13,0.13,$\cdots$,0.13)   & (0.12,0.01,$\cdots$,0.01)           & (0.33,0.01,$\cdots$,0.01)           \\
10 & (0.11,0.11,$\cdots$,0.11)   & (0.14,0.01,$\cdots$,0.01)           & (0.32,0.01,$\cdots$,0.01)           \\
11 & (0.10,0.10,$\cdots$,0.10)   & (0.16,0.01,$\cdots$,0.01)           & (0.30,0.01,$\cdots$,0.01)            \\ \hline
\end{tabular}
\centering
\end{table}

\begin{table}[H]
\caption{  Choices of $\{A_j\}$ for $f\in W^{4}_{\infty}([0,1]^{7})$}
\begin{tabular}{cllll}
\hline
$n$ & \multicolumn{1}{c}{Function}         & \multicolumn{1}{c}{1st-order derivatives} & \multicolumn{1}{c}{2nd-order derivatives}  \\ \hline
3 & (1.10,1.10,$\cdots$,1.10) & (2.11,0.01,$\cdots$,0.01) & (0.240,0.01,$\cdots$,0.01) \\
4 & (1.10,1.10,$\cdots$,1.10) & (1.55,0.01,$\cdots$,0.01) & (0.355,0.01,$\cdots$,0.01) \\
5 & (1.10,1.10,$\cdots$,1.10) & (1.31,0.01,$\cdots$,0.01) & (0.470,0.01,$\cdots$,0.01) \\
6 & (0.22,0.22,$\cdots$,0.22) & (1.00,0.01,$\cdots$,0.01) & (0.630,0.01,$\cdots$,0.01) \\
7 & (0.37,0.37,$\cdots$,0.37) & (0.46,0.01,$\cdots$,0.01) & (0.690,0.01,$\cdots$,0.01) \\
8 & (0.10,0.10,$\cdots$,0.10) & (0.40,0.01,$\cdots$,0.01) & (0.700,0.01,$\cdots$,0.01) \\ \hline
\end{tabular}
\centering
\end{table}

\begin{table}[H]
\caption{  Choices of $\{A_j\}$ for $f\in W^{4}_{\infty}([0,1]^{10})$}
\begin{tabular}{clll}
\hline
$n$ & \multicolumn{1}{c}{Function}                                  & \multicolumn{1}{c}{1st-order derivatives}                 & \multicolumn{1}{c}{2nd-order derivatives}                \\ \hline
2 & (0.880,0.880,$\cdots$,0.880) & (0.090,0.01,$\cdots$,0.01)  & (1.04,0.01,$\cdots$,0.01) \\
3 & (0.650,0.650,$\cdots$,0.650) & (1.190,0.01,$\cdots$,0.01)  & (1.10,0.01,$\cdots$,0.01)  \\
4 & (0.595,0.595,$\cdots$,0.595) & (2.020,0.01,$\cdots$,0.01)  & (2.10,0.01,$\cdots$,0.01) \\
5 & (0.983,0.983,$\cdots$,0.983) & (2.118,0.01,$\cdots$,0.01) & (2.08,0.01,$\cdots$,0.01) \\
6 & (0.100,0.100,$\cdots$,0.100) & (0.185,0.01,$\cdots$,0.01) & (1.80,0.01,$\cdots$,0.01) \\ \hline
\end{tabular}
\centering
\end{table}
\begin{table}[H]
\caption{ Choices of $\{A_j\}$ for non-periodic function approximations}
\begin{tabular}{cclclc}
\hline
&$n$ & \quad \quad $d=7$ & $n$ & \quad \quad $d=10$            \\ \hline
& 4 & (0.05,$\cdots$,0.05)   & 2 & (1.08,$\cdots$,1.08)    \\
& 5 & (0.05,$\cdots$,0.05) & 3 & (0.72,$\cdots$,0.72)   \\
& 6 & (0.05,$\cdots$,0.05) & 4 & (0.28,$\cdots$,0.28) \\
& 7 & (0.05,$\cdots$,0.05)  & 5 & (0.61,$\cdots$,0.61) \\
 & 8 & (0.05,$\cdots$,0.05)    & 6 & (0.40,$\cdots$,0.40)   \\ \hline
\end{tabular}
\centering
\end{table}
 \end{document}